\documentclass[11pt]{amsart}

\usepackage{enumitem}
\usepackage{amsmath}%
\usepackage{amsfonts}%
\usepackage{amssymb}%
\usepackage{graphicx}
\usepackage{mathrsfs}
\usepackage{filecontents}
\usepackage{hyperref}
\usepackage{fullpage}
\usepackage{mathtools}
\newtheorem{theorem}{Theorem}
\theoremstyle{plain}

\newtheorem{claim}[theorem]{Claim}

\newtheorem{conjecture}[theorem]{Conjecture}

\newtheorem{lemma}[theorem]{Lemma}

\newtheorem{problem}[theorem]{Problem}
\newtheorem{proposition}[theorem]{Proposition}

\newtheorem{thm}[theorem]{Theorem}

\numberwithin{equation}{section}
\numberwithin{theorem}{section}
\numberwithin{case}{section}

\numberwithin{subcase}{case}

\newcommand{\msc}[1]{\begin{center}MSC2000: #1.\end{center}}
\newcommand{\key}[1]{\begin{center}Keywords: #1.\end{center}}

\def\B{\mathcal{B}}

\def \a{\alpha}
\def\d{\delta}
\def \e{\varepsilon}
\def \eps{\e}
\def \r{\gamma}
\def\calP{\mathcal{P}}

\def\COMMENT#1{}
\let\COMMENT=\footnote

\begin{document}

\title{Exact minimum codegree threshold for $K^- _4$-factors}
\author{Jie Han, Allan Lo, Andrew Treglown and Yi Zhao}
\thanks{The first author is supported by FAPESP (Proc. 2014/18641-5). The third author is supported by EPSRC grant EP/M016641/1. The fourth author is partially supported by NSF grant DMS-1400073.}
\date{\today}

\begin{abstract}
Given hypergraphs $F$  and $H$, an $F$-factor in $H$ is a set of vertex-disjoint copies of $F$ which cover all the vertices in $H$. Let $K^- _4$ denote the $3$-uniform hypergraph with $4$ vertices and $3$ edges. We show that for sufficiently large $n\in 4 \mathbb N$, every $3$-uniform hypergraph $H$ on $n$ vertices with minimum codegree at least $n/2-1$ contains a $K^- _4$-factor. Our bound on the minimum codegree here is best-possible. It resolves a conjecture of Lo and Markstr\"om~\cite{LoMa} for large hypergraphs, who earlier proved an asymptotically exact version of this result.
Our proof makes use of the absorbing method as well as a result of Keevash and Mycroft~\cite{mycroft} concerning almost perfect matchings in hypergraphs.

\end{abstract}

\maketitle
 \msc{5C35, 5C65, 5C70}
\key{Tiling, Hypergraphs, Absorbing method}
\section{Introduction}\label{sec}
Given two hypergraphs $H$  and $F$, an \emph{$F$-tiling} in $H$ 
is a collection of vertex-disjoint copies of $F$ in $H$. An
$F$-tiling is called \emph{perfect} if it covers all the vertices of $H$.
Perfect $F$-tilings are also referred to as \emph{$F$-factors} or \emph{perfect $F$-packings}. 
Note that perfect $F$-tilings are generalisations of perfect matchings (which correspond to the case when $F$ is a single edge).

Tiling problems  have been widely studied for graphs. The seminal Hajnal--Szemer\'edi theorem~\cite{hs} states that every graph $G$ on $n \in r\mathbb N$ vertices and with $\delta (G) \geq (1-1/r)n$ contains a  $K_r$-factor.
More generally, given any graph $F$, K\"uhn and Osthus~\cite{kuhn2} determined, up to an additive constant, the minimum degree threshold that forces a  $F$-factor in a graph. 
See~\cite{survey} for a survey including many of the results on graph tiling.

Given a $k$-uniform hypergraph ($k$-graph for short) $H$  with a $d$-element vertex set $S$ (where $0 \leq d \leq k-1$) we define the \emph{degree} $\deg_H (S)$ of $S$ in $H$ to be the number of edges containing $S$. The \emph{minimum $d$-degree $\delta _{d} (H)$} of $H$ is the minimum of $\deg_H (S)$ over all $d$-element sets of vertices in $H$. 
We also refer to  $\delta _1 (H)$ as the \emph{minimum vertex degree} of $H$ and  $\delta _{k-1}(H)$ the \emph{minimum codegree} of $H$.

In recent years there have been significant efforts on finding minimum $d$-degree conditions that force a perfect matching in a $k$-graph. For example, for every $k \geq 3$, R\"odl, Ruci\'nski and Szemer\'edi~\cite{rrs} determined the minimum codegree threshold that forces a sufficiently large $k$-graph $H$ to contain a perfect matching. 
Other than the matching problems, only a few hypergraph tiling problems have been studied -- most of them were done recently.

Given a $k$-graph $F$ of order $f$ and an integer $n$ divisible by $f$, we define the threshold ${\delta_d(n,F)}$ 
 as the smallest integer $t$ such that every $n$-vertex $k$-graph $H$ with $\delta_{d}(H)\ge t$ contains an $F$-factor. We simply write $\d(n, F)$ for $\d_{k-1}(n, F)$.
One of the earliest results on hypergraph tiling was given by K\"{u}hn and Osthus \cite{KuOs-hc}, who proved that $\d(n, C_4^3)= n/4+ o(n)$, where $C_4^3$ is the (unique) 3-graph with four vertices and two edges.
Later Czygrinow, DeBiasio, and Nagle \cite{CDN} showed that for sufficiently large $n\in \mathbb{N}$, $\d(n, C_4^3)=n/4+1$ if $n\in 8\mathbb{N}$ and $\d(n, C^3_4)= n/4$ otherwise.
Let $K_4^3$ denote the complete 3-graph on four vertices. Lo and  Markstr\"om~\cite{LM1} showed that $\d(n, K^3_4)= 3n/4 + o(n)$. Independently and simultaneously Keevash and Mycroft \cite{mycroft} proved that for sufficiently large $n\in 4\mathbb{N}$, $\d(n, K^3_4)= 3n/4-2$ if $n\in 8\mathbb{N}$ and $\d(n, K^3_4)= 3n/4-1$ otherwise.
More recently Han and Zhao \cite{HZ3} and independently Czygrinow \cite{Czy} determined $\d_1(n, C_4^3)$ exactly for sufficiently large $n$. 
Mycroft \cite{Myc} determined $\d(n, F)$ asymptotically for many $k$-partite $k$-graphs $F$ (including complete $k$-partite $k$-graphs and loose cycles).
One of these thresholds, $\d(n, C_6^3)$, where $C_6^3$ denotes the 3-uniform loose cycle on 6 vertices, was determined exactly by Gao and Han \cite{GH_C6} very recently.
Han, Zang, and Zhao \cite{HZZ} determined $\d_1(n, K)$ asymptotically for all complete $3$-partite $3$-graphs $K$. 
See the surveys~\cite{rrsurvey, zsurvey} for detailed overviews of matching and tiling problems in hypergraphs.

Let $K^- _4$ denote the $3$-graph with $4$ vertices and $3$ edges. Lo and Markstr\"om~\cite{LoMa} proved that 
$n/2 -1 \le \d(n, K_4^-)\le n/2 + o(n)$. Let us recall the construction that gives the lower bound. Given two disjoint vertex sets $A, B$, define $\mathcal B[A,B]$ to be the $3$-graph on $A\cup B$ whose edge set consists of all those triples that contain an odd number of vertices from $A$. 
Suppose that $n \equiv 0 \mod 4$. If $n \not \equiv 0 \mod 3$ and $|A|=|B|=n/2$, we have that $\delta _2 (\mathcal B[A,B]) =n/2-2$ but $\mathcal B[A,B]$ does not contain a  $K^- _4$-factor.
If $n  \equiv 0 \mod 3$ and $|A|=n/2+1$, $|B|=n/2-1$, again we have that $\delta _2 (\mathcal B[A,B]) =n/2-2$ but $\mathcal B[A,B]$ does not contain a  $K^- _4$-factor. (See Proposition~1 in~\cite{LoMa} for details.)

In this paper we determine $\d(n, K_4^-)$ exactly for sufficiently large $n$, thereby resolving a conjecture of Lo and Markstr\"om~\cite{LoMa}.

\begin{thm}\label{mainthm}
There exists an $n_0 \in \mathbb N$ such that the following holds. Suppose that $H$ is a $3$-graph on $n \geq n_0$ vertices where $n$ is divisible by $4$. If $\delta _2 (H) \geq n/2-1$ then $H$ contains a  $K^- _4$-factor.  Thus, $\d(n, K_4^-)= n/2 -1$ for $n\in 4\mathbb{N}$ and $n\ge n_0$.
\end{thm}
The proof of Theorem~\ref{mainthm} makes use of the \emph{absorbing method} -- a technique that was first used in~\cite{rrs2} and has  subsequently been
applied to numerous embedding problems in extremal graph theory. We also apply a result of Keevash and Mycroft~\cite{mycroft} concerning almost perfect matchings in hypergraphs.

The paper is organised as follows. 
In the next section we derive Theorem~\ref{mainthm}  from three main lemmas, preceded by an overview of the proof and a comparison with the proof in \cite{LoMa}.
We give some useful tools in Section~\ref{secuse}. We prove an {almost} perfect tiling lemma in Section~\ref{secalmost} and an {absorbing} lemma in Section~\ref{secabs}. The {non-extremal} case is tackled in Section~\ref{secex}.


\section{Notation and proof of Theorem~\ref{mainthm}}\label{secnotation}
\subsection{Notation}
Given a set $X$ and  $r\in \mathbb N$, we write $\binom{X}{r}$ for the set of all $r$-element subsets  of $X$.
For simplicity, given vertices $x_1, \dots, x_t$ and a set of vertices $S$, we often write $x_1\cdots x_t$ for $\{x_1,\dots, x_t\}$ and $S\cup x_1$ for $S\cup \{x_1\}$.

Let $H$ be a $3$-graph.
We write $V(H)$ for the vertex set and $E(H)$ for the edge set  of $H$.
Define $e(H):=|E(H)|$. 
We denote the \emph{complement of $H$} by $\overline{H}$. That is, $\overline{H} := (V(H), \binom{V(H)}{3}\setminus E(H))$.
Given $x,y \in V(H)$,  we write $N_H(xy)$  to denote the \emph{neighborhood of $xy$}, that is, the family of those vertices in $V(H)$ which, together with $x,y$, form an edge in $H$. 
If $X \subseteq V(H)$ we write $N_H (xy,X):= N_H (xy)\cap X$, and $\deg _H (xy,X):= |N_H (xy,X)|$. For this and similar notation, we often omit the subscript if the underlying hypergraph is clear from the context.

Given $X \subseteq V(H)$, we write $H[X]$ for the \emph{subhypergraph of $H$ induced by $X$},  namely, $H[X] := (X, E(H)\cap \binom{X}{3})$. We write $e_H (X)$ or simply $e(X)$ for $e(H[X])$. In addition, we let $H \setminus X := H[V(H) \setminus X]$.
If $K$ is a spanning subgraph of $H[X]$, then we say that $X$ \emph{spans a copy of $K$ in $H$}. In particular, this does not necessarily mean that $X$ \emph{induces} a copy of $K$ in $H$. When counting the number of copies of $K$ in $H$, we only count the number of subsets of $V(H)$ that span copies of $K$ in $H$. For example, we say that $K_4^3$ only contains one copy of $K_4^-$ (instead of four copies).

Let $\gamma> 0$ and $H$ and $H'$ be two 3-graphs on the same vertex set $V$. We say that $H$ \emph{$\gamma$-contains} $H'$ if $|E(H')\setminus E(H)| \le \gamma |V|^3$, namely, $H$ misses at most $\gamma |V|^3$ edges from $H'$. 
Given  $\gamma>0$, we call a 3-graph $H=(V, E)$ on $n$ vertices \emph{$\gamma$-extremal} if there is a partition of $V= A\cup B$ such that $|A|= \lfloor n/2 \rfloor $, $|B|= \lceil n/2 \rceil$ and $H$ $\gamma$-contains $\mathcal B[A,B]$.

For any $x\in V(H)$, we define the \emph{link graph} $L_x$ to be the graph with vertex set $V(H)\setminus \{x\}$ and where $yz \in E(L_x)$ if and only if $xyz \in E(H)$.
Let $G$ be a graph, $X \subseteq V(G)$ and $x \in V(G)$. We define
$N_G (x), e_G(X), G[X], N_G (x,X), \deg _G (x,X)$ analogously to the $3$-graph case.
We write $\delta (G)$ for the \emph{minimum degree of $G$} and $\Delta (G)$ for
the \emph{maximum degree of $G$}. Given disjoint $X,Y \subseteq V(G)$, we write
$e_G (X,Y)$ for the number of edges in $G$ with one endpoint in $X$ and the other endpoint in $Y$. 

Throughout the paper, we write $0<\alpha \ll \beta \ll \gamma$ to mean that we can choose the constants
$\alpha, \beta, \gamma$ from right to left. More
precisely, there are increasing functions $f$ and $g$ such that, given
$\gamma$, whenever we choose $\beta \leq f(\gamma)$ and $\alpha \leq g(\beta)$, all
calculations needed in our proof are valid. 
Hierarchies of other lengths are defined in the obvious way.

\subsection{Overview of the proof of Theorem~\ref{mainthm}}
In the next subsection we will combine the three main lemmas of the paper to prove Theorem~\ref{mainthm}. Before this we give an overview of the proof.
It is instructive to first describe the strategy used in~\cite{LoMa} to prove the asymptotic version of Theorem~\ref{mainthm}.

Let $0< \eps \ll \gamma \ll \eta$ and $n$ be sufficiently large. Suppose that $H$ is a  $3$-graph on $n$ vertices where $\delta _2 (H) \geq (1/2+\eta)n$.
The proof in~\cite{LoMa} splits into two main tasks.
\begin{itemize}
\item {\bf Step 1 (Absorbing set):} Find an \emph{absorbing} set $W \subseteq V(H)$ such that $|W|\leq \gamma n$. $W$ has the property that given \emph{any} set $U \subseteq V(H) \setminus W$ where $U \in 4 \mathbb N$ and $|U|\leq  \eps n$, both
$H[W]$ and $H[W \cup U]$ contain $K^- _4$-factors. 

\item {\bf Step 2 (Almost tiling):} Let $H':= H \setminus W$. Find a $K^- _4$-tiling $\mathcal K$ in $H'$ that covers all but at most $\eps n$ vertices.
\end{itemize}
Note that after Steps 1 and 2 one immediately obtains a $K^- _4$-factor in $H$. Indeed, let $U:=V(H')\setminus V(\mathcal K)$. Then $H[W \cup U]$ contains a $K^- _4$-factor $\mathcal K'$ and so $\mathcal K \cup \mathcal K'$ is a $K^- _4$-factor in $H$.

To show that $H$ contains the desired absorbing set $W$, Lemma~1.1 in~\cite{LoMa} implies that it suffices to show that $H$ is \emph{closed}. Roughly speaking, $H$ is closed if, for any $x,y \in V(H)$, there are many small sets $S \subseteq V(H)$ such that
both $H[S\cup x]$ and $H[S\cup y]$ contain $K^- _4$-factors (see Section~\ref{secabs} for the formal definition).
Using that $\delta _2 (H) \geq (1/2+\eta)n$, it is not too difficult to show that there is a partition of $V(H)$ into at most three parts such that each of these partition classes is closed.
So a key task in~\cite{LoMa} is to `merge' these closed classes into a single closed class. 
For this, it suffices to show that are many `bridges' between the partition classes (see Lemma~\ref{lem:bridge}): An $(X,Y)$-bridge is a triple $(x,y,S)$ where $x \in X$, $y \in Y$ and $S \subseteq V(H)$ such that 
$H[S\cup x]$ and $H[S\cup y]$ contain $K^- _4$-factors. This is precisely the strategy used in~\cite{LoMa} to prove that $H$ is closed, and thus contains an absorbing set $W$.
A short argument then shows that, since $\delta _2 (H') \geq (1/2+\eta/2)n$, $H'$ contains an almost perfect $K^- _4$-tiling, as desired.

We now turn to our proof of Theorem~\ref{mainthm}. Let $H$ be a sufficiently large $3$-graph on $n$ vertices where $\delta _2 (H) \geq n/2-1$. 
If $H$ is close to the extremal example $\mathcal B[A,B]$ then it is not clear whether one can find an absorbing set in $H$. Indeed, let $H^*:=\mathcal B[A,B]$ where $|A|=|B|=n/2$. 
Suppose that $U \subseteq B$ where $|U|=4$. Consider any $W \subseteq V(H^*)$ such that
$H^*[W]$ contains a $K^- _4$-factor. Then it is easy to see that $|W \cap B| \equiv 0 \mod 3$. However, for any such set $W$, $H^*[W \cup U]$ does not contain a $K^- _4$-factor as $|(W\cup U) \cap B| \equiv 1 \mod 3$.

Thus, in the case when $H$ is close to the extremal example $\mathcal B[A,B]$ we do not use the absorbing method. Instead, in Section~\ref{secex}, we give a direct argument to show that $H$ contains a $K^- _4$-factor.
In the case when $H$ is non-extremal we follow Steps~1 and~2 as above. However, since we now only have that $\delta _2 (H) \geq n/2-1$, the argument becomes significantly more involved.

To find an absorbing set when $H$ is non-extremal we again show that $H$ is closed. Suppose that there exists $x \in V(H)$ such that there are very few edges $abc \in E(H)$ so that
$a b c x$ spans a copy of $K^- _4$ in $H$. In this case we give a direct argument to show that $H$ contains an absorbing set (see Lemma~\ref{lem:ve}).
Otherwise, we show that our minimum codegree ensures that $V(H)$ can be partitioned into at most \emph{four} sets 
such that each is closed in $H$ (see Lemma~\ref{lem:P}).
We again merge these sets into a single closed class by finding many bridges between the sets. For this, we use that if $H$ is non-extremal then in any partition $A,B$ of $V(H)$ with $|A|,|B|\geq n/5$, we have many edges that intersect $A$ in precisely $1$ vertex
and many edges that intersect $A$ at precisely $2$ vertices (see Lemma~\ref{lem:C}). The process of proving that non-extremal 3-graphs $H$ are closed is quite involved and forms the heart of the paper (most of Section~\ref{secabs} is devoted to this task).

In Section~\ref{secalmost} we tackle Step~2 for non-extremal $3$-graphs $H$. Our lower minimum codegree condition means that we cannot use the argument from~\cite{LoMa} here. Instead, we translate the problem to one on almost perfect matchings in hypergraphs.
We then (somewhat carefully) apply a result of Keevash and Mycroft~\cite{mycroft} to obtain an almost perfect matching in some auxiliary hypergraph whose $4$-edges correspond to copies of $K^- _4$ in $H'$.
Thus, we obtain an almost perfect $K^-_4$-tiling in $H'$.

\subsection{Proof of Theorem~\ref{mainthm}}\label{secsketch}
As outlined in the previous subsection, the proof of Theorem~\ref{mainthm} consists of three main parts: the extremal case; obtaining an absorbing set in the non-extremal case; and finding an almost perfect tiling in the non-extremal case.

Our first lemma deals with  the latter part. In fact, it implies that $H$ has an almost perfect $K_4^-$-tiling even if $\delta_2(H)$ is (slightly) less than $n/2$.

\begin{lemma}\label{lem:almost}
Let $1/n \ll \phi \ll \gamma \ll 1$.
Let $H$ be a $3$-graph on $n$ vertices with $\delta_{2}(H)\ge (1/2 - \gamma) n$.
Then $H$ contains a $K_4^-$-tiling covering all but at most $\phi n$ vertices.
\end{lemma}

The next result yields the absorbing set in the non-extremal case.

\begin{lemma} \label{lem:abs}
Let $1/n \ll \phi \ll \eps \ll \r \ll 1$.
Let $H$ be a $3$-graph of order~$n$ with $\delta_{2}(H)\ge (1/2 - \r) n$. 
Suppose that $H$ is not $3\gamma$-extremal.
Then there exists an absorbing set $W\subseteq V(H)$ of order at most $\e n$ so that for any $U \subseteq V(H) \setminus W$ with $|U|\le \phi n$ and $|U|\in 4\mathbb{N}$, both $H[W]$ and $H[U\cup W]$ contain $K^- _4$-factors.
\end{lemma}

If $H$ is extremal, then we will find a  $K_4^-$-factor using the following lemma.

\begin{lemma}\label{extremallemma}
There exist $\gamma >0$ and $n_0 \in \mathbb N$ such that the following holds. Suppose that $H$ is a $3$-graph on $n \geq n_0$ vertices where $n$ is divisible by $4$.
If $\delta _2 (H) \geq n/2-1$ and $H$ is $\gamma$-extremal, then $H$ contains a  $K^- _4$-factor.
\end{lemma}

Theorem~\ref{mainthm} now follows easily from Lemmas~\ref{lem:almost}--\ref{extremallemma}.

\begin{proof}[Proof of Theorem~\ref{mainthm}]
Let $1/n \ll \phi \ll \eps \ll \r \ll 1$ with $n \in 4 \mathbb{N}$.
Let $H$ be a $3$-graph of order $n$ with $\delta_2(H) \ge n/2 - 1$.
If $H$ is $3\gamma$-extremal, then by Lemma~\ref{extremallemma} $H$ contains a  $K^- _4$-factor.

Therefore, we may assume that $H$ is not $3\gamma$-extremal.
By Lemma~\ref{lem:abs}, there exists an absorbing set $W\subseteq V(H)$ of order at most $\e n$ so that for any $U \subseteq V(H) \setminus W$ with $|U|\le \phi n$ and $|U|\in 4\mathbb{N}$, both $H[W]$ and $H[U\cup W]$ contain $K^- _4$-factors.
Let $H ':= H \setminus W$. 
Note that $n':=  |H '| \ge (1- \e) n $ and  $\delta_2(H') \ge n/2 -1 - \e n \ge (1/2- 2 \e )n'$.
By Lemma~\ref{lem:almost}, $H'$ contains a $K_4^-$-tiling $\mathcal{M}_1$ covering all but at most $\phi n'$ vertices.
Let $U := V(H') \setminus V (\mathcal{M}_1)$.
Since $|U| \le \phi n' \le \phi n $, $H[U\cup W]$ contains a  $K^- _4$-factor~$\mathcal{M}_2$.
Then $\mathcal{M}_1 \cup \mathcal{M}_2$ is a  $K_4^-$-factor in~$H$, as desired.
\end{proof}

\section{Useful results}\label{secuse}
We will need the following result, which follows immediately from a theorem of Baber and Talbot~\cite[Theorem~2.2]{BaberTalbot} and the \emph{supersaturation} phenomenon discovered by Erd{\H{o}}s and Simonovits~\cite{ErdosSimonovits}.

\begin{proposition}\label{prop:0.3}
There exist a constant $c'>0$ and an integer $n'$ such that every 3-graph $H$ of order $n \ge n'$ with $e(H)> 0.3\binom n3$ contains at least $c' n^4$ copies of $K_4^-$.
\end{proposition}

Let $H$ be a $3$-graph of order~$n$. 
For any 3-set $T\subseteq V(H)$, let $L_H(T)$, or simply $L(T)$, be the set of vertices $v$ such that $H[T\cup v]$ contains a copy of $K^- _4$.
If $T = xyz\in E(H)$, then $L(T)= (N(xy) \cap  N(yz)) \cup (N(xy) \cap  N(xz)) \cup (N(yz) \cap  N(xz))$. 
The following proposition gives a lower bound on such $|L(T)|$.

\begin{proposition}\cite[Proposition~2.1]{LoMa}\label{prop:le}
Let $H$ be a 3-graph of order $n$.
Then for every edge $e=xyz$ and any $U \subseteq V(H)$, $
|L(e) \cap U|\ge (\deg (xy, U) + \deg (yz, U) + \deg(xz, U) - |U|)/2$.
\end{proposition}

Let $H$ be a $3$-graph and let $V_1,V_2,V_3 \subseteq V(H)$.
We say that an edge $v_1v_2v_3 \in E(H)$ is an \emph{$V_1V_2V_3$-edge} if $v_i \in V_i$ for all $i \in [3]$.
We denote by $e_H(V_1V_2V_3)$ the number of $V_1V_2V_3$-edges.
The following simple result will be  applied in the proof of the non-extremal case of Theorem~\ref{mainthm}. We remark that the property guaranteed by Lemma~\ref{lem:C} is in fact the only property of non-extremalness that will be used in the proof
of Lemma~\ref{lem:abs} (and thus in the entire proof of Theorem~\ref{mainthm}).

\begin{lemma}\label{lem:C} Let $0<1/n \ll \gamma <1/100$.
Suppose that $H$ is a $3$-graph of order $n$ where $\delta _2 (H) \geq (1/2-\gamma )n$.
Let $X, Y$ be any bipartition of $V(H)$ where $|X|,|Y|\geq n/5$. If $H$ is not $3\r$-extremal, then there exist at least $\r^2 n^3$ $XXY$-edges and at least $\r^2 n^3$ $XYY$-edges.
\end{lemma}

\begin{proof}
Suppose that $H$ contains fewer than $\r^2 n^3$ $XXY$-edges. We will show that $H$ is $3\r$-extremal. (The case when $H$ contains fewer than $\r^2 n^3$ $XYY$-edges is analogous.)
We have $\sum_{x, x'\in X}|N(x x')| = \sum_{x, x'\in X}|N(x x', X)| + \sum_{x, x'\in X}|N(x x', Y)|$ and $\sum_{x, x'\in X}|N(x x', Y)|\le \r^2 n^3$. Since $\delta_2(H)\ge (1/2 - \r)n$, we have $\sum_{x, x'\in X}|N(x x')|\ge \binom{|X|}2 (1/2 -\r)n$. So we get $3e(X)=\sum_{x, x'\in X}|N(x x', X)|\ge \binom{|X|}2 (1/2 -\r)n - \r^2 n^3$. 
Therefore, $e(X)\ge \frac13\binom{|X|}2 (1/2 - 2\r)n$ and in particular, $|X|\ge (1/2 - 2\r)n$.

Since $\sum_{x\in X, y\in Y} |N(x y)| = 2e(XXY) + 2e(XYY)$, we derive that $e(XYY)\ge \frac12|X| |Y|(1/2 - \r)n - \r^2 n^3$ similarly. Since $|X|, |Y|\ge n/5$, we get $e(XYY)\ge \frac12|X| |Y|(1/2 - 2\r)n$ and in particular, $|Y|\ge (1/2 - 2\r)n$.

Therefore we have $(1/2 - 2\r)n \le |X|, |Y| \le (1/2 + 2\r)n$. This implies that
\[
e(X)\ge \frac13\binom{|X|}2 (1/2 - 2\r)n \ge \frac13\binom{|X|}2 \frac{(1/2 - 2\r)}{(1/2 + 2\r)}|X|\ge (1-8\r)\binom{|X|}3.
\]
A similar calculation shows $e(XYY)\ge (1-8\r) |X| \binom{|Y|}2$. This implies that $| E(\B[X, Y])\setminus E(H) |\le 8\r \binom{n}3$. 
After moving at most $2\gamma n$ vertices from $X$ to $Y$ or from $Y$ to $X$, we obtain a bipartition $X', Y'$ of $V(H)$ such that $|X'|= \lfloor n/2 \rfloor$, 
$|Y'|= \lceil n/2 \rceil$. Then $| E(\B[X', Y']) \setminus E(\B[X, Y]) | \le 2\gamma n \binom{n-1}{2}$. 
Consequently 
\begin{align*}
| E(\B[X', Y']) \setminus E(H) | &\le | E(\B[X', Y']) \setminus E(\B[X, Y]) | + | E(\B[X, Y]) \setminus E(H)| \\
&\le 2\gamma n \binom{n-1}{2} + 8\r \binom{n}3 < 3\r n^3.
\end{align*}
Therefore $H$ is $3\r$-extremal.
\end{proof}

\begin{proposition}\label{prop:axy}
Let $\beta >0$ and $H$ be a $3$-graph of order~$n$.
Let $X, Y$ be a partition of $V(H)$.
Suppose that there are at most $\beta \binom{|X|}2 \binom{|Y|}2$ copies of $K_4^-$ with two vertices in $X$ and two vertices in $Y$. 
Then 
\[
(|Y| - 1)e(XXY) + (|X|-1)e(XYY) \le 2(1+ \beta)\binom{|X|}2 \binom{|Y|}2.
\]
\end{proposition}

\begin{proof}
Observe that $\sum e(S)=(|Y| - 1)e(XXY) + (|X|-1)e(XYY)$, where the sum is over all 4-sets $S$ such that $|S\cap X|=|S\cap Y|=2$. 
Our assumption is that there are at most $\beta \binom{|X|}2 \binom{|Y|}2$ 4-sets $S$ such that $|S\cap X|=|S\cap Y|$ and $e(S)\ge 3$. So we have  $\sum e(S) \le 2\binom{|X|}2 \binom{|Y|}2 + 2\beta \binom{|X|}2 \binom{|Y|}2= 2(1+ \beta)\binom{|X|}2 \binom{|Y|}2$.
\end{proof}

\section{Almost $K_4^-$-tiling}\label{secalmost}

In this section, we prove Lemma~\ref{lem:almost} which implies that any $3$-graph $H$ with minimum codegree slightly less than that in Theorem~\ref{mainthm} must contain an almost perfect $K^- _4$-tiling.

The key tool in the proof of Lemma~\ref{lem:almost} is a result of Keevash and Mycroft~\cite{mycroft} on almost perfect matchings in hypergraphs. Before we can state this result, 
we need the following terminology. 
For an integer~$k$, a \emph{$k$-system} is a hypergraph $J$ in which every edge of $J$ has size at most $k$ and $\emptyset \in E(J)$. We call an edge of size $s$ in $J$ an \emph{$s$-edge}.
Let $J_s$ be the $s$-graph on $V(J)$ induced by all $s$-edges of~$J$.
The \emph{minimum $r$-degree} of $J$, denoted by $ \hat{\delta}_r(J)$, is the minimum $\deg_{J_{r+1}}(e)$ among all $e\in E(J_r)$. (Note that this is different from $\d_r(J_{r+1})$, which is the minimum $\deg_{J_{r+1}}(S)$ among all $r$-sets $S\subseteq V(J)$.)
The \emph{degree sequence of~$J$} is $\hat{\delta}(J) = ( \hat{\delta}_0(J) , \hat{\delta}_1(J), \dots, \hat{\delta}_{k-1}(J))$. Given $a_0, a_1, \dots, a_{k-1} \geq 0$ we write $\hat{\delta}(J) \geq (a_0, a_1, \dots , a_{k-1})$ to mean that
$\hat{\delta}_i(J) \geq a_i$ for all $i$.

We will apply the following special case of Lemma~7.6 from~\cite{mycroft}.

\begin{lemma} \label{lem:keevash}
Suppose that $1/n \ll \phi \ll \r \ll \beta, 1/k$.
Let~$V$ be a set of size $n$.
Suppose that $J$ is a $k$-system on $V$ such that 
\begin{enumerate}
	\item[\rm (i)] $\hat{\delta} (J) \ge ( n, (\frac{k-1}{k} - \r )  n , (\frac{k-2}{k} - \r ) n , \dots,  (\frac{1}{k} - \r ) n)$ and
	\item[\rm (ii)] for any $p \in [k-1]$ and set $S \subseteq V$ with $S = \lfloor pn/k \rfloor$, we have $e(J_{p+1}[S] )\ge \beta n^{p+1}$.
\end{enumerate}
Then $J_k$ contains a matching $M$ which covers all but at most $\phi n $ vertices of $J$.
\end{lemma}

We now prove Lemma~\ref{lem:almost} by defining a $4$-system $J$ such that 
\begin{align}\label{eq:Js}
E(J_0) = \{ \emptyset \}, E(J_1) =  V(H), E(J_2) = \binom{V(H)}{2} , E(J_3) = E(H), \text{ and } E(J_4) = \mathcal{K}_{4}^-(H),
\end{align}
where $\mathcal{K}_{4}^-(H)$ denotes the set of $4$-tuples in $V(H)$ that span a copy of $K_4^-$ in~$H$.
If $J$ satisfies the hypothesis of Lemma~\ref{lem:keevash}, then we are done.
Otherwise, we deduce some structural properties of~$H$, update $J$ appropriately and apply Lemma~\ref{lem:keevash} again.

\begin{proof}[Proof of Lemma~\ref{lem:almost}]
Define $1/n \ll \phi \ll \r \ll \beta  \ll c'$ where $c'$ is the constant  from Proposition~\ref{prop:0.3}. Let $H$ be as in the statement of the lemma and $V := V(H)$.
Define a $4$-system~$J$ as in \eqref{eq:Js}.
By Proposition~\ref{prop:le}, $\hat{\delta}_3(J) \ge (1/4- 3 \r /2) n $.
Hence $\hat{\delta}(J) \ge ( n , n-1, (1/2 - \r) n, (1/4 - 3\r/2) n)$.
If Lemma~\ref{lem:keevash}(ii) holds, then $J_4$ contains a matching~$M$ which covers all but $\phi n$ vertices of~$J$ thereby proving the lemma. 

Thus we may assume that there exist some $p \in [3]$ and a set $S \subseteq V$ with $S = \lfloor pn/4 \rfloor$ where $e(J_{p+1}[S]) < \beta n^{p+1}$. We note that $p \ne 1, 3$. Indeed, for any $S$ of size $\lfloor n/4\rfloor$, we have $e(J_2[S])= \binom{|S|}2> \beta n^2$. Now consider  any $S \subseteq V$ such that $|S|  = \lfloor 3 n / 4 \rfloor$.
Note that $\delta_2( H [ S ] ) \ge  \delta_2(H) - \lceil n/4 \rceil \ge (1/4 - 2\r ) n \gg 0.3 |S|$.
Proposition~\ref{prop:0.3} therefore implies that $e(J_{4}[S]) \ge c' |S|^4  \ge \beta  n^4$.
So we must have that $p = 2$.

Let $S \subseteq V$ be such that $|S| = \lfloor n /2 \rfloor$ and $e(J_{3}[S]) = e(H[S]) < \beta n^3$. In general, for any 
set $U\subseteq V$ with $|U| = \lfloor n /2 \rfloor$ and $e(H[S]) < \beta n^3$, we call a pair $xy \in \binom{U}{2}$ \emph{$U$-good} if $\deg_{ H [U] } (xy) \le 3 \beta^{1/2} n $ (and so $\deg_{H} (xy, V \setminus U) \ge ( 1/2 - \r - 3 \beta^{1/2} ) n $).
Then at most $\beta^{1/2} n^2$ pairs $xy \in \binom{U}{2}$ are \emph{not} $U$-good.
We call a triple $xyz \in \binom{U}{3}$ \emph{$U$-good} if every pair in $xyz$ is $U$-good.  Note that
\begin{enumerate}
\item[(a)] for any $W\subseteq U$, at least $\binom{|W|}{3} - \beta^{1/2} n^2 (|W|-2)$ triples of $W$ are $U$-good;
\item[(b)] for any $U$-good triple $T=xyz$, $|L(T)| \ge	| \left( N_H( xy ) \cap N_H( xz ) \cap N_H( yz ) \right) \cap (V \setminus  U) | \ge (1/2 - 10 \beta ^{1/2} ) n$.
\end{enumerate}

Define a $4$-system~$J'$ obtained from~$J$ by adding all $S$-good triples that are not edges of $H$ to the edge set.
Because of (b), we have $\deg_{J'}(T)\ge (1/2 - 10 \beta ^{1/2} ) n\ge (1/4 - 3\gamma /2) n$. Thus $\hat{\delta}(J') \ge ( n , n-1, (1/2 - \r) n, (1/4 - 3\gamma /2) n)$.
Since $J \subseteq J'$,  $J'$ satisfies Lemma~\ref{lem:keevash}(ii) for $p = 1, 3$.
If $J'$ also satisfies Lemma~\ref{lem:keevash}(ii) for $p=2$, then Lemma~\ref{lem:keevash} gives a matching in $J'_4= J_4$ which covers all but $\phi n$ vertices, proving the lemma.
Otherwise, there exists $S' \subseteq V$  such that $|S'| = \lfloor n /2 \rfloor$ and $e(J_{3}'[S']) < \beta  n^3$.
We claim that $|S' \cap S| \le 3 \beta  ^{1/4} n$ -- otherwise by (a), the number of $S$-good triples in $S'\cap S$ is at least
\[
\binom{3 \beta ^{1/4}  n}3 - \beta^{1/2} n^2\cdot (3\beta ^{1/4}n-2)  > \beta   n^3,
\]
implying that $e(J_{3}'[S']) > \beta  n^3$, a contradiction.

Define a $4$-system~$J^*$ obtained from $J'$ by adding all $S'$-good triples that are not in $J'$.
Once again, we have $\hat{\delta}(J^*) \ge ( n , n-1, (1/2 - \r) n, (1/4 - 3\gamma /2) n)$ by (b).
Since $J \subseteq J^*$,  $J^*$ satisfies Lemma~\ref{lem:keevash}(ii) for $p = 1, 3$.
Consider a set $S^*\subseteq V$ with $|S^*| = \lfloor n /2 \rfloor$. 
As $|S| = |S'| = \lfloor n /2 \rfloor$ and $|S' \cap S| \leq 3\beta  ^{1/4} n$, we have $|S'\cup S|\ge n - 3\beta^{1/4} n -1$. 
Thus $S^*$ contains at least $n/6$ vertices from $S$ or at least $n/6$ vertices from $S'\setminus S$. In either case, since $J^*$ contains all $S$-good and $S'$-good triples, we have $e(J^*_3[S^*])\ge \binom{n/6}{3} - \beta^{1/2} n^2 \cdot n/6> \beta n^3$.
So $J^*$ satisfies Lemma~\ref{lem:keevash}(ii) for $p = 2$.
Therefore $J_4 = J^*_4$ contains a matching~$M$ which covers all but $\phi n$ vertices, proving the lemma.
\end{proof}

\section{The absorbing lemma}\label{secabs}
In this section
we prove Lemma~\ref{lem:abs} which is an absorbing result for the case when $H$ is not $3\gamma$-extremal.
For this, we need the following terminology.
Let $H$ be a $3$-graph of order~$n$.
Given an integer $c\ge 1$ and vertices $x, y\in V(H)$, we say that the vertex set $S\subseteq V(H)$ is an \emph{$(x,y)$-connector of length~$c$} if $S\cap \{x, y\}=\emptyset$, $|S|=4c-1$ and both $H[S\cup x]$ and $H[S\cup y]$ contain $K_4^-$-factors.
Given an integer $c\ge 1$ and a constant $\eta>0$, two vertices $x,y \in V(H)$ are \emph{$(c,\eta)$-close} to each other if there exist at least $\eta n^{4c-1}$ $(x, y)$-connectors of length~$c$ in $H$.
For $x \in V(H)$, we denote by $\tilde{N}_{c,\eta}(x)$ the set of vertices $y$ in $H$ that are $(c,\eta)$-close to~$x$.
A subset $U\subseteq V(H)$ is said to be \emph{$(c,\eta)$-closed in $H$} if any two vertices in $U$ are $(c, \eta)$-close to each other. If $V(H)$ is $(c,\eta)$-closed in $H$ then we simply say that \emph{$H$ is $(c,\eta)$-closed}.

Given $X \subseteq V(H)$,  $X$ being $(c,\eta)$-closed in $H$  is \emph{not} the same notion as $H[X]$ being  $(c,\eta)$-closed. Indeed, the former implies
that between any $x,y \in X$ there are at least $\eta n^{4c-1}$ $(x, y)$-connectors of length~$c$ in $H$. On the other hand, the latter implies
that between any $x,y \in X$ there are at least $\eta |X|^{4c-1}$ $(x, y)$-connectors of length~$c$ in $H[X]$.

Given an integer $c \ge 1$ and $X, Y\subseteq V(H)$, a triple $(x, y , S)$ is an \emph{$(X, Y)$-bridge of length $c$} if $x\in X$, $y\in Y$ and $S$ is an $(x, y)$-connector of length~$c$.

We will apply the following two results from~\cite{LM1}. The first, a special case of
Lemma~1.1 from~\cite{LM1}, states that if $H$ itself is $(c,\eta)$-closed then
$H$ contains a small absorbing set.

\begin{lemma}\cite{LM1} \label{lem:abs3}
Let $0 < 1/ n  \ll \phi \ll \eps \ll \eta , 1/c$.
Let $H$ be a $3$-graph of order~$n$.
Suppose  that $H$ is $(c,\eta)$-closed.
Then there exists an absorbing set $W\subseteq V(H)$ of order at most $\eps n$ such that $|W| \in 4\mathbb{N}$ and for any $U \subseteq V(H) \setminus W$ such that $|U|\le \phi n$ and $|U|\in 4\mathbb{N}$, $H[W]$ and $H[U\cup W]$ have $K_4^-$-factors.
\end{lemma}
The next result is a special case of Proposition~2.1 from~\cite{LM1}.
\begin{proposition}\cite{LM1}\label{prop21}
Let $0 < 1/n \ll \eta' \ll  \eta , \eps ,1/c$ with $c \in \mathbb{N}$.
Suppose $H$ is a $3$-graph of order $n$ and there exists a vertex $x\in V(H)$ with $|\tilde{N}_{c,\eta}(x)|\ge \e n$.
Then $\tilde{N}_{c, \eta}(x) \subseteq \tilde{N}_{c+1, \eta'}(x)$.
\end{proposition}
The proof of the next simple result is similar to that of
Lemma~2.2 from~\cite{LM1} (so we omit it). It states that if one has two disjoint `closed' sets $X$ and $Y$ in $H$ and $H$ contains many $(X,Y)$-bridges, then in fact $X\cup Y$ is closed.
\begin{lemma} \label{lem:bridge}
Let $0 < 1/n \ll \eta' \ll  \eta , \eps ,1/c,1/p$ with $c,p \in \mathbb{N}$.
Let $H$ be a $3$-graph of order~$n$.
Let $X,Y\subseteq V(H)$ be disjoint such that both $X$ and $Y$ are $(c, \eta)$-closed in $H$.
Suppose further there exist at least $\eps n^{4p+1}$ $(X, Y)$-bridges of length~$p$.
Then $X\cup Y$ is $(2c+p,\eta')$-closed in $H$.
\end{lemma}

The following result gives another condition that ensures we can `merge' two closed sets $V_1, V_2$ into a larger closed set $V_1 \cup V_2$.

\begin{lemma}\label{lem:tran}
Let $0 < 1/n \ll \eta' \ll  \eta , \eps ,1/c$ with $c \in \mathbb{N}$.
Let $H$ be a $3$-graph of order~$n$.
Let $V_1,\dots, V_d$ be disjoint subsets of $V(H)$ such that $2 \le d \le 4$ and each $V_i$ is $(c, \eta)$-closed in~$H$.
Let $a_1, \dots, a_d$ be non-negative integers such that $a_1 \ge 1$ and $\sum a_i = 4$.
Suppose there exist at least $\eps n^4$ copies $F$ of $K_4^-$ in $H$ such that 
$|V(F)\cap V_i|=a_i$ for all $i\in [d]$ and there exist at least $\eps n^4$ copies $F'$ of $K_4^-$ in $H$ such that $|V(F')\cap V_1|=a_1-1$, $|V(F')\cap V_2|=a_2+1$ and $|V(F')\cap V_j|=a_j$ for all $3\le j\le d$.
Then $V_1\cup V_2$ is $(5c+1, \eta')$-closed in $H$.
\end{lemma}

\begin{proof}
Let $\eta''$ be such that $\eta' \ll \eta'' \ll \eta, \eps, 1/c$.
Consider any vertex-disjoint copies $F, F'$ of $K_4^-$ in~$H$ such that 
$|V(F)\cap V_i|=a_i$ for all $i\in [d]$ and $|V(F')\cap V_1|=a_1+1$, $|V(F')\cap V_2|=a_2-1$ and $|V(F')\cap V_j|=a_j$ for all $3\le j\le d$.
Note that there are at least $(\eps n^4/2)^2$ choices of $(F,F')$.
Let $V(F) = \{ x,x_1,x_2, x_3 \}$ and $V(F') = \{ y,y_1,y_2, y_3 \}$ such that $x \in V_1$, $y \in V_2$ and for each $j \in [3]$, $x_j,y_j \in V_{i_j}$ for some~$i_j$.
For each $j \in [3]$, $V_{i_j}$ is $(c, \eta)$-closed.
Therefore, there exist $S_1,S_2,S_3$ such that each $S_j$ is an $(x_j, y_j)$-connector of length~$c$ and $V(F), V(F'), S_1,S_2,S_3$ are vertex-disjoint. 
Note that there are at least $ ( \eta n^{4c-1}/2)^3$ choices of~$(S_1,S_2,S_3)$.
Set $S:=\{x_1, x_2, x_3, y_1, y_2, y_3\} \cup S_1\cup S_2\cup S_3$.
Note that $S$ is a $(x,y)$-connector of length~$3c+1$.
Indeed, $H[x\cup S] \supseteq F \cup  \bigcup_{j \in [3]} H[ y_j \cup S_j]$ has a $K_4^-$-factor, and similarly $H[y\cup S]$ has a $K_4^-$-factor.
Thus $(x, y, S)$ is a $(V_1, V_2)$-bridge of length~$3c+1$.
Hence, we have at least
\[
\frac1{(4(3c+1)+1)!} \left( \frac{\eps n^4}{2} \right)^2 \left( \frac{\eta n^{4c-1}}{2} \right)^3
\ge \eta'' n^{4(3c+1)+1}
\]
$(V_1, V_2)$-bridges of length $3c+1$. 
By Lemma~\ref{lem:bridge}, $V_1\cup V_2$ is $(5c+1, \eta')$-closed in~$H$.
\end{proof}

\subsection{There exists a vertex $v \in V(H)$ such that $v \in L(e)$ for very few edges $e \in E(H)$.} \label{sec:onevertex}

Let $H$ be a $3$-graph satisfying the hypothesis of Lemma~\ref{lem:abs}.
Suppose further that there exists a vertex $v \in V(H)$ such that there are at most $\eps n^3$ edges $e$ such that $v \in L(e)$ (that is, $e \cup v$ spans at least three edges).
In Lemma~\ref{lem:ve}, we show that there exists a small set $V_0 \subseteq V(H)$ such that $H[V_0]$ contains a $K_4^-$-factor and $H \setminus V_0$ is $(6,\eta _*)$-closed for some constant $\eta _*>0$. 
First we will need the following result for graphs.

\begin{proposition} \label{prop:partition}
Let $n \in \mathbb{N}$ and $0 < \gamma \le 1/20$.
Let $G$ be a graph of order~$n$ with $(1/2 - \gamma) n \le  \delta(G) \le \Delta(G) \le 3n/5 $.
Suppose that $|N(x) \triangle N(y) | \le \gamma n$ for every edge $xy\in E(G)$.
Then there exists a bipartition $X , Y$ of $V(G)$ such that $\delta(G[X]), \delta(G[Y]) \ge (1/2 - 5 \gamma) n $ and $(1/2- 5 \gamma) n \le |X|, |Y|  \le (1/2 + 5 \gamma) n $.
\end{proposition}

\begin{proof}
Let $X, Y$ be a bipartition of $V(G)$ such that $e(X,Y)$ is minimised. 
First we show that $e(X,Y) \le 3 \gamma n^2 /5 $.
Consider a vertex $x_0 \in V(G)$.
Let $X_0 := N(x_0) \cup x_0$ and $Y_0 := V(G) \setminus X_0$. 
Note that
\begin{align}
	e(X,Y) \le e(X_0,Y_0)  \le \sum_{x \in N(x_0) } |N(x) \triangle N(x_0) |  \le \deg(x_0) \gamma n \le 3 \gamma n^2 /5 , \label{eqn:eXY}
\end{align}
as claimed. 

Suppose there exists $v \in V(G)$ such that $\deg(v,X), \deg(v,Y) > 4 \gamma n$.
Without loss of generality, assume that $\deg(v,X) \ge \deg(v) /2 \ge n/5$. 
For each $w \in N(v,X)$, we have 
\begin{align*}
\deg(w,Y) \ge |N(w)\cap N(v) \cap Y| \ge \deg(v,Y)  - |N(v)\triangle N(w)| > 3 \gamma n, 
\end{align*}
as $|N(v) \triangle N(w)|\le \r n$.
Thus $e(X,Y) > 3 \deg(v,X)  \gamma n \ge 3 \gamma n^2/5$ contradicting~\eqref{eqn:eXY}.
Therefore, for all $v \in V(G)$, either $\deg(v,X) \le 4 \gamma n$ or $\deg(v,Y) \le 4 \gamma n$.
Since $e(X,Y)$ is minimal, we have $\delta(G[X]), \delta(G[Y]) \ge (1/2 - 5 \gamma)n$.
\end{proof}

Define $\eta_*>0$ to be the constant $\eta '$ obtained by applying Lemma~\ref{lem:tran} with $1/30$, $1/128$ and $1$ playing the roles of $\eta$, $\eps$ and $c$, respectively.
\begin{lemma}\label{lem:ve}
For $0 < 1/n\ll \e\ll\r \ll 1$, let $H$ be a $3$-graph of order~$n$ with $\delta_{2}(H)\ge (1/2 - \r) n$.
Suppose there exists a vertex $v\in V(H)$ such that there are less than $\e n^3$ edges $e\in E(H)$ such that $v \in L(e)$.
Then there exists  $V_0\subseteq V(H)$ of order at most $8\sqrt[4]\e n$ such that $H[V_0]$ contains a $K_4^-$-factor and $H\setminus V_0$ is $(6, \eta_* )$-closed.
\end{lemma}

Here is a sketch of our proof.
First we show that there exists a partition $X,Y,V_0$ such that $H[V_0]$ contains a $K_4^-$-factor and almost all $XXY$- and $YYX$-edges exist.
We then show that both $X$ and $Y$ are $(1, 1/30)$-closed in $H[X \cup Y]$.
Furthermore, we show that there are many copies $F,F'$ of $K_4^-$ such that $|V(F) \cap X|=3$, $|V(F) \cap Y|=1$ and $|V(F') \cap X|=2 = |V(F') \cap Y|$.
Then by Lemma~\ref{lem:tran}, $H[X \cup Y]$ is $(6, \eta_*)$-closed.

\begin{proof}[Proof of Lemma~\ref{lem:ve}]
Set $G:= L_{v}$.
Then $\delta_{2}(H)\ge (1/2 - \r) n$ implies that $\delta(G)\ge (1/2 - \r) n$.
We say that an edge $uw \in E(G)$ is \emph{good} if $|N_G(u)\cap N_H(uw)|\le 3\sqrt{\e} n$ and $|N_G(w)\cap N_H(uw)|\le 3\sqrt{\e} n$, otherwise we call it \emph{bad}.

\begin{claim}\label{clm:ve} \
\begin{enumerate}[label = {\rm (\roman*)}]
\item There are at most $\sqrt{\e} n^2$ bad edges in $G$.
\item If $u \in V(G)$ is incident with a good edge, then $\deg_{G} (u)  \le (1/2 + \gamma + 3\sqrt{\e} )  n$.
\end{enumerate}
\end{claim}

\begin{proof}[Proof of claim]
For each bad edge $uw$, there are at least $3\sqrt{\e} n$ edges $e$ of $H$ such that $uw \subseteq e$ and $e$ contains at least two edges in~$G$.
Moreover, $v \in L(e)$.
Thus there are at least $3\sqrt{\e} n$ edges~$e$ of $H \setminus \{v\}$ such that $uw \subseteq e$ and $v \in L(e)$.
If there are at least $\sqrt\e n^2$ bad edges, then there are at least $\frac13 \sqrt\e n^2 \cdot 3\sqrt\e n=\e n^3$ edges $e \in E(H)$ such that $v \in L(e)$, contradicting the assumption.
Thus (i) holds. 

Suppose that $u \in V(G)$ is incident with a good edge $uw$ in $G$.
Note that 
\begin{align*}
 3 \sqrt{\eps} n \ge |N_G(u)\cap N_H(uw)| \ge \deg_{G} (u) + \deg_H(uw) - n.
\end{align*}
Since $\delta_2(H) \ge (1/2- \r ) n$, (ii) holds. 
\end{proof}

Let $V_0'$ be the set of vertices $u \in V(G)$ that are incident to at least $\sqrt[4]\e n$ bad edges in $G$.
Since there are at most $\sqrt{\e} n^2$ bad edges, $|V_0'| \le 2\sqrt[4]\e n$.
By the greedy algorithm and Proposition~\ref{prop:le}, there exists a vertex set $V_0 \supseteq V_0'$ such that $H[V_0]$ contains a $K_4^-$-factor and $|V_0| \le 8\sqrt[4]\e n$. 
By pairing $v$ with a vertex in $V'_0$, we can ensure that $v \in V_0$. 

Set $V' := V(H) \setminus V_0$ and $n' := |V'|$.
Let $H' := H [V']$, so $\delta_2(H') \ge (1/2- \r - 8\sqrt[4]\e) n \ge (1/2- 2 \r) n'$.
Let $G'$ be the spanning subgraph of $G \setminus V_0$ induced by the good edges.
Note that for all $u \in V'$,
\begin{align}
	\deg_{G'}(u)  \ge \delta(G) - |V_0| - \sqrt[4]\e n \ge (1/2-\r - 9 \sqrt[4]\e)  n  \ge (1/2 - 2\gamma )  n' \nonumber 
\end{align}
and by Claim~\ref{clm:ve}(ii), $\deg_{G'}(u) \le (1/2 + \gamma + 3\sqrt{\e} )  n \le (1/2 + 2\gamma )  n' $.
Since each edge $uw$ in $G'$ is good, we have 
\begin{align}
|N_{G'}(u)\cap N_{H'}(uw)|\le 3\sqrt{\e} n \le 4\sqrt{\e} n'
\label{eqn:L'H'}
\end{align}
and so $| N_{G'}(u)\cup N_{H'}(uw)| \ge 2(1/2 - 2\r) n' - 4\sqrt{\e} n'$.
Consequently
\begin{align*}
|N_{G'} (u) \setminus N_{G'}(w) | & = |(N_{G'}(u)\setminus N_{G'}(w))\setminus N_{H'}(uw) | + | (N_{G'}(u) \setminus N_{G'}(w)) \cap N_{H'}(uw) | \\
& \le |V(H')\setminus (N_{G'}(w)\cup N_{H'}(uw))| + |N_{G'}(u)\cap N_{H'}(uw)|\\
&\le (n' - 2(1/2- 2\r)n' +4 \sqrt\e n') + 4 \sqrt\e n' \le 5\r n'.
\end{align*}
Similarly, we have $|N_{G'}(w)\setminus N_{G'}(u)| \le 5\r n'$.
Thus $|N_{G'}(u)\triangle N_{G'}(w)|\le 10 \r n'$.

Applying Proposition~\ref{prop:partition} to~$G'$ we obtain a bipartition $X \cup Y$ of $V'$ such that $(1/2- 50 \gamma) n' \le |X|, |Y|  \le (1/2 + 50 \gamma) n' $ and $\delta(G'[X]) \ge (1/2 - 50 \gamma) n'  \ge (1 - 200\r) |X|$ and $\delta(G'[Y])  \ge (1 - 200\r) |Y|$.

Consider any edge $uw$ in $G'[X]$.
Observe that 
\[
|X \setminus N_{G'}(u)|\le |X| - \deg_{G'}(u, X)\le  (1/2 + 50 \gamma) n' - (1/2 - 50 \gamma) n'  \le  100 \r n'.
\]
So by~\eqref{eqn:L'H'},
\begin{align}
	\deg_{H'}(uw,Y) & =  \deg_{H'}(uw) - \deg_{H'}(uw,X) \nonumber \\
	& \ge 	
	(1/2- 2 \r) n' - ( |X \setminus N_{G'}(u)| + |N_{G'}(u)\cap N_{H'}(uw)| ) \nonumber \\
	& \ge ( 1/2 - 2\r - 100 \r - 4\sqrt{\e} ) n' \nonumber \\
	& \ge (1/2 - 103 \r) n' \ge (1-400 \r)|Y|, \label{eqn:d2H'}
\end{align}
as $|Y|  \le (1/2 + 50 \gamma) n' $. Therefore, 
\begin{align}
e_{H'}(XXY) & \ge \sum_{uw \in E(G'[X])} \deg_{H'}(uw,Y) \ge 
e(G'[X]) (1-400 \r)|Y| \nonumber \\
& \ge (1 - 200\r)(1-400\r ) |X|^2|Y|/2
\ge (1 - 600\r) |X|^2|Y|/2 \label{eqn:XXY}
\end{align}
as $\delta(G'[X])\ge (1 - 200\r) |X|$.

Note that for every 3-set $T\subseteq X$ that forms a triangle in $G'[X]$, \eqref{eqn:d2H'} implies that
\begin{align}
|L(T) \cap Y|\ge (1 - 1200 \r ) |Y|. \label{eqn:L(T)}
\end{align}
Recall that $\delta(G'[X]) \ge (1 - 200\r) |X|$.
So there are at least $(1-200\r)(1-400\r) \binom{|X|}3 (1-1200\r) |Y| \ge  n^{4}/128$ copies of $K_4^-$ in $H'$ with three vertices in~$X$ and one in~$Y$.

We now show that $X$ is $(1, 1/30)$-closed in $H'$. 
Let $x, x'$ be two distinct vertices in $X$.
Since $\delta(G'[X]) \ge (1 - 200\r) |X|$, there are at least $(1-400\r )|X| (1- 600 \r) |X|/2 \ge n^2/10$ choices of $x_1,x_2 \in X$ such that both $x x_1 x_2$ and $x' x_1 x_2$ form triangles in~$G'$. 
Let $Z: = L_{H'}(x x_1 x_2) \cap L_{H'}(x' x_1 x_2) \cap Y$.
Thus $|Z| \ge (1 - 2400 \r ) |Y| \ge n/3$ by~\eqref{eqn:L(T)}.
Notice that for each $z \in Z$, $x_1 x_2 z$ is an $(x,x')$-connector of length~$1$ in $H'$.
Thus, we get at least $n^3/30$ $(x,x')$-connectors of length~$1$, that is, $x$ and $x'$ are $(1,1/30)$-close in~$H'$.
Therefore, $X$ is $(1, 1/30)$-closed in~$H'$ as required.

By a similar argument, we have that $Y$ is $(1, 1/30)$-closed in~$H'$ and $e_{H'}(XYY) \ge (1 - 600\r) |X||Y|^2/2 $.
Together with \eqref{eqn:XXY}, we have 
\begin{align*}
	(|Y|-1)e_{H'}(XXY) + (|X|-1)e_{H'}(XYY)  \ge 4(1 - 600\r)\binom{|X|}2 \binom{|Y|}2. 
\end{align*}
By Proposition \ref{prop:axy}, there are at least $(1-1200\r)\binom{|X|}2\binom{|Y|}2\ge n^4/128$ copies of $K_4^-$ in $H'$ with two vertices in each of $X$ and~$Y$.
Recall that there are at least $n^{4}/128$ copies of $K_4^-$ in $H'$ with three vertices in~$X$ and one in~$Y$.
Lemma~\ref{lem:tran} implies that $X\cup Y$ is $(6, \eta_*)$-closed in~$H'$, as desired.
\end{proof}

\subsection{Partitioning $V(H)$ into $(c, \eta)$-closed components.}
Because of Lemma~\ref{lem:ve}, we may assume that for every $v \in V(H)$, there are at least $\eps n^3$ edges $e$ such that $v \in L(e)$.
Recall that $\tilde{N}_{c, \eta}(v)$ is the set of vertices that are $(c, \eta)$-close to~$v$ in $H$. 
First we show that $\tilde{N}_{1, \eta}(v)$ is large for each $v \in V(H)$.

\begin{proposition}\label{prop:Nv}
Let $n \in \mathbb{N}$ and $0 < \eps, \r <1$.
Let $H$ be a $3$-graph of order~$n$ with $\delta_{2}(H)\ge (1/2 - \r) n$. 
Let $v \in V(H)$.
Suppose that there are at least $\e n^3$ edges $e \in E(H)$ such that $v \in L(e)$.
Then $|\tilde{N}_{1, \r \e}(v)|\ge (1/4 - 3\r)n$.
\end{proposition}

\begin{proof}
Let $E' \subseteq E(H)$ be the set of edges $e$ such that $v \in L(e)$.
By Proposition \ref{prop:le}, for every edge $e \in E'$, $|L(e)|\ge (1/4 - 2\r)n$. 
Thus, we have $$\sum_{ e \in E'} |L(e)| \ge |E'| (1/4-2\r)n.$$
For any $e \in E'$ and any $u \in L(e) \setminus \{v\}$, $e$ is an $(u,v)$-connector of length~$1$. 
Hence, if $u \ne v$ is a vertex in $V(H)$ and there are at least $\r\e n^3$ edges $e \in E'$ such that $u \in L(e)$, then $ u\in \tilde{N}_{1,\r\e}(v)$. 
Thus
\[
\sum_{e \in E'} |L(e)| \le |\tilde{N}_{1, \r\e}(v)| |E'| + n \cdot \r\e n^3 \le (|\tilde{N}_{1, \r\e}(v)| + \gamma n ) |E'|
\]
as $|E'|\ge \e n^3$.
Therefore, we have $|\tilde{N}_{1,\r\e}(v)| \ge (1/4-3\r)n$.
\end{proof}

Next we show that $V(H)$ can be partitioned into at most $4$ parts such that each part is $(16, \eta)$-closed in $H$.

\begin{lemma}\label{lem:P}
Let $0< 1/n \ll \eta \ll \eps , \r\ll 1$.
Suppose $H$ is a $3$-graph of order~$n$ such that $|\tilde{N}_{1, \r \eps}(v)|\ge (1/4 - 3\r)n$ for every $v\in V(H)$.
Then there is a partition $\mathcal{P}=\{V_1,\dots, V_d\}$ of $V(H)$ such that $d\le 4$ and each $V_i$ is $(16, \eta)$-closed in $H$ and $|V_i|\ge (1/4 - 4\r)n$.
\end{lemma}

\begin{proof}
Let $\alpha, \eta_0, \eta_1, \eta_2, \eta_3, \eta_4$ be such that $\alpha \ll \gamma$,  $\eta_0 := \e\r$ and 
\begin{align*}
1/n \ll \eta = \eta_{4}\ll \eta_3 \ll \eta_2 \ll  \eta_1 \ll \eta_0, \alpha.
\end{align*}
Throughout this proof, for $v\in V(H)$ and $i\in [4]$, we write $\tilde{N}_{2^i, \eta_i}(v)$ as $\tilde N_{2^i}(v)$ for short. 
By assumption, for any $v\in V(H)$, $|\tilde N_{2^0}(v)|=|\tilde{N}_{1, \eta_0}(v)| \ge (1/4 - 3\r)n$.
We also write $2^i$-close (respectively $2^i$-closed) for $(2^i, \eta_i)$-close (respectively $(2^i, \eta_i)$-closed).
By Proposition \ref{prop21} and the choice of the $\eta_i$s, we may assume that $\tilde {N}_{2^i}(v)\subseteq \tilde {N}_{2^{i+1}}(v)$ for all $0\le i\le 3$ and all $v\in V(H)$.
Hence, if $W\subseteq V(H)$ is $2^i$-closed in $H$ for some $i\le 4$, then $W$ is $ 2^{4}$-closed in $H$.

Since $|\tilde{N}_{ 2^0}(v)| \ge (1/4 - 3\r)n$ for any $v\in V(H)$, any set of five vertices in $V(H)$ contains two vertices $u,v$ such that $|\tilde{N}_{2^0}(u)\cap \tilde{N}_{2^0}(v)| \ge (5(1/4 - 3\r)n - n)/\binom{5}{2} \ge n/50$.
Thus the number of $(u,v)$-connectors of length $2$ in $H$ is at least
\begin{align*}
n/50\cdot (\eta_0 n^3 - n^2) \cdot (\eta_0 n^3 - 4 n^2) \ge \eta_1 n^7,
\end{align*}
which implies that $u$ and $v$ are $2^1$-close to each other in $H$.
Also we may assume that there are two vertices that are not $2^4$-close to each other, as otherwise $H$ is $2^4$-closed and lemma holds with~$\calP=\{V(H)\}$.

Let $d$ be the largest integer such that
there exist $v_1,\dots, v_{d}\in V(H)$ such that no pair of them are $2^{6-d}$-close to each other. 
Note that $d$ exists by our assumption and $2\le d\le 4$ by our observation. 
Fix such $v_1,\dots, v_{d}\in V(H)$, by Proposition \ref{prop21}, we can assume that any two of them are not $2^{5-d}$-close to each other.
Consider $\tilde N_{2^{5-d}}(v_i)$ for all $i\in [d]$. We have the following facts.
\begin{enumerate}[label={\rm(\roman*)}]
\item Any $v\in V(H)\setminus\{v_1,\dots, v_{d}\}$ must be in $\tilde N_{2^{5-d}}(v_i)$ for some $i\in [d]$. 
\item For any $i\ne j$, $|\tilde N_{2^{5-d}}(v_i)\cap \tilde N_{2^{5-d}}(v_j)|<\a n$.
\end{enumerate}
Note that if (i) fails for some $v \in V(H)$, then $v, v_1,\dots, v_{d}$ contradicts the definition of~$d$.
If (ii) fails with $|\tilde N_{2^{5-d}}(v_1)\cap \tilde N_{2^{5-d}}(v_2)| \ge \a n$ say, then the number of $(v_1, v_2)$-connectors $S$ of length $2^{6-d}$ of the form $S = z \cup S_1 \cup S_2$, where $z \in \tilde N_{2^{5-d}}(v_1)\cap \tilde N_{2^{5-d}}(v_2)$ and for $i=1,2$, $S_i$ is a $(v_i,z)$-connector of length~$2^{5-d}$, is at least 
\[
\frac{1}{ (2^{6-d}4-1)!}\a n \left(\frac{\eta_{5-d}n^{2^{5-d}4-1}}2\right)^2 \ge \eta_{6-d} n^{2^{6-d}4-1}.
\]
This implies that $v_1$ and $v_2$ are $2^{6-d}$-close in $H$, a contradiction. 

For each $i\in [d]$, let $U_i:= ( \tilde N_{2^{5-d}}(v_i)\cup v_i) \setminus \bigcup_{j\in [d]\setminus \{i\}} \tilde N_{2^{5-d}}(v_j)$.
Note that each $U_i$ is $2^{5-d}$-closed in $H$.
Indeed, if there exist $u_1, u_2 \in U_i$ that are not $2^{5-d}$-close to each other, then $\{u_1, u_2\}\cup (\{v_1,\dots, v_{d}\}\setminus\{v_{i}\})$ contradicts the definition of $d$.

Let $U_0 : =V(H)\setminus (U_1\cup\cdots \cup U_{d})$.
By (i) and~(ii), we have $|U_0|\le \binom{d}{2}\a n$.
We partition $U_0$ into $U_{0,1}, \dots, U_{0,d}$ as follows.
For each $v \in U_0$, since $|\tilde N_{2^0}(v)\setminus U_0|\ge (1/4 - 3\r)n - |U_0| \ge d\a n$, there exists $i\in [d]$ such that $|\tilde N_{2^0}(v) \cap U_i| \ge \a n$.
In this case we add $v$ to $U_{0,i}$ (we add $v$ to an arbitrary $U_{0,i}$ if there are more than one such $i$). 

Let $\mathcal{P} = \{V_1, \dots, V_d\}$ be the partition of $V(H)$ such that $V_i = U_i \cup U_{0,i}$ for all $i \in [d]$.
Let $ \eta_{6-d} \ll \eta'' \ll \eta' \ll \eta_{5-d}$.
Consider any $i \in [d]$ and any $v \in U_{0,i}$.
Since $|\tilde N_{2^0}(v)\cap U_i| \ge \a n$, it is easy to see that $U_i \cup v$ is $(2^{5-d} + 1, \eta')$-closed.
Similarly, $U_i \cup \{v,v'\}$ is $(2^{5-d} + 2, \eta'')$-closed for any $v,v' \in U_{0,i}$.
Thus each $V_i$ is $(2^{5-d}+2, \eta'')$-closed, so $2^{4}$-closed by Proposition~\ref{prop21}. 
Further, $|V_i| \geq |U_i| \geq (1/4-3\gamma)n - \binom{d}{2}\alpha n \geq (1/4-4\gamma)n$.
\end{proof}

\subsection{Proof of Lemma~\ref{lem:abs}}

By Lemma~\ref{lem:P}, there is a partition $\mathcal{P}=\{V_1,\dots, V_d\}$ of $V(H)$ such that $d\le 4$ and each $V_i$ is $(16, \eta)$-closed in $H$ and $|V_i| \ge (1/4-4 \gamma) n$.
Our new goal is to find many $(V_i,V_j)$-bridges for some $i \ne j$ so that we can apply Lemma~\ref{lem:bridge} to reduce that number of $(c, \eta')$-closed components of~$V(H)$.
The following lemma is crucial in proving Lemma~\ref{lem:abs}; we defer its proof to Section~\ref{sec:B5}.

For any $U_1,U_2,U_3,U_4 \subseteq V(H)$, we say that a copy $F$ of $K_4^-$ is \emph{of type $U_1U_2U_3U_4$} if there is an ordering $v_1, v_2,  v_3, v_4$ of $V(F)$ such that $v_i \in U_i$ for all $i \in [4]$.

\begin{lemma}\label{lem:B5}
Let $0 < 1/n \ll  \eta' \ll \eps, \eta \ll \gamma \le \rho \ll 1/c$.
Let $H$ be a $3$-graph of order~$n$ with $\delta_2(H) \ge (1/2 - \gamma)n$.
Suppose that there is a partition $\mathcal{P} = \{V_1, \dots, V_d\}$ of $V(H)$ such that $2\leq d \le 4$
 and each $V_i$ is $(c, \eta)$-closed in $H$ and $|V_i| \ge (1/4-4 \gamma )n $.
Suppose that there are not necessarily distinct $i_1, i_2, i_3, i_4\in [d]$ such that $\{i_1, i_2\} \cap \{i_3, i_4\}= \emptyset$ and at least $\eps n^4$ copies of $K_4^-$ are of type $V_{i_1} V_{i_2} V_{i_3} V_{i_4}$.
Then there exist distinct $i, j \in [d]$ such that $V_{i} \cup V_{j}$ is $(3c+1,\eta')$-closed in $H$.
\end{lemma}

\begin{proof}[Proof of Lemma~\ref{lem:abs}]
Let $\eta_*$ be as defined before Lemma~\ref{lem:ve}.
Since $\eps \ll \r \ll 1$, we may assume that $\eps \ll \eta_*$.
Let $\eps_0$ be such that $\phi \ll \eps_0 \ll \eps$.
Let $H$ be a $3$-graph of order~$n$ with $\delta_2(H) \ge (1/2 - \gamma) n$ 
and so that $H$ is not $3\gamma$-extremal.

If there exists a vertex $v \in V(H)$ such that there are less than $\eps_0 n^3$ edges $e \in E(H)$ such that $v \in L(e)$, then Lemma~\ref{lem:ve} implies that there exists $V_0 \subseteq V(H)$ of order at most $8 \sqrt[4]{\eps_0} n $ such that $H[V_0]$ contains a $K_4^-$-factor and $H \setminus V_0$ is $(6,\eta_*)$-closed.
Apply Lemma~\ref{lem:abs3} to $H \setminus V_0$ and obtain an absorbing set $W' \subseteq V(H) \setminus V_0$ of order at most $\eps n/2$ such that for any $U \subseteq V(H) \setminus ( V_0 \cup W')$ such that $|U|\le \phi n$ and $|U|\in 4\mathbb{N}$, both $H[W']$ and $H[U\cup W']$ have $K_4^-$-factors.
Lemma~\ref{lem:abs} holds by setting $W : = W' \cup V_0$.

Therefore, we may assume that for every $v \in V(H)$, there are at least $\eps_0 n^3$ edges $e \in E(H)$ such that $v \in L(e)$.
Let $c_4 := 16$ and for $d \in [3]$, let $c_d := 5c_{d+1}+1$.
Let $\eta, \eta_1, \dots, \eta_4$ be such that $\phi  \ll \eta_1 \ll \eta_2 \ll \eta_3 \ll \eta_4 \ll \eta \ll \eps_0$.
Let $d \le 4$ be the smallest integer such that there is a partition $\mathcal{P}=\{V_1,\dots, V_d\}$ of $V(H)$ such that $d\le 4$ and each $V_i$ is $(c_d, \eta_d)$-closed and $|V_i|\ge (1/4 - 4\r)n$.
Note that $d$ exists by Proposition~\ref{prop:Nv} and Lemma~\ref{lem:P}.
If $d = 1$, then Lemma~\ref{lem:abs} holds by Lemma~\ref{lem:abs3}.
We will now show that we obtain a contradiction if $d=2,3,4$ and so $d=1$, as required.

\medskip
\noindent \textbf{Case 1: $d \in \{3,4\}$.}
\medskip

Without loss of generality, assume $|V_1|\le \cdots \le |V_d|$.
Note that we have $(1/2 - 8 \gamma) n\le |V_1| + |V_2|\le 2n/3$. 

First assume that $e(V_1 V_2 V_j)\ge \eta n^3$ for some $j\ge 3$.
Since $\delta_2(H) \ge (1/2 - \gamma)n$, by Proposition~\ref{prop:le}, each $V_1 V_2 V_j$-edge is contained in at least $(1/2 - 3 \gamma)n/2$ copies of $K_4^-$. By averaging, there exists $i\in [d]$ such that there are at least 
\[
\frac14 \eta n^3 \left(\frac12 - 3 \gamma \right) \frac{n}2 \cdot  \frac12 \ge \frac{\eta}{40} n^4
\] 
copies of $K_4^-$ of type $V_1 V_2 V_j V_i$ ($\frac12$ appears because a copy of $K_4^-$ may be counted twice if $i\in \{1, 2, j\}$).
It is easy to see that we can order $1, 2, j, i$ as $i_1, i_2, i_3, i_4$ such that $\{i_1, i_2\} \cap \{i_3, i_4\}= \emptyset$. By Lemma~\ref{lem:B5} and Proposition~\ref{prop21}, there exist distinct $i, j \in [d]$ such that $V_{i} \cup V_{j}$ is $(c_{d-1}, \eta_{d-1})$-closed in $H$. 
Also, by Proposition~\ref{prop21}, every set in $\{V_1,\dots , V_d\} \setminus \{
V_{i} , V_{j} \}$ is $(c_{d-1}, \eta_{d-1})$-closed in $H$. 
Altogether, this contradicts the minimality of~$d$.

So we may assume that $e(V_1 V_2 V_j)< \eta n^3$ for all $j\ge 3$.
Since
\[	\sum_{x\in V_1, y\in V_2} |N(xy)|   = 	2e(V_1 V_1 V_2) + 2e(V_1 V_2 V_2) +  \sum_{j=3}^d e(V_1V_2V_j),  
\]	
it follows that  $|V_1| |V_2| (1/2 - \gamma ) n  \le 2e(V_1 V_1 V_2) + 2e(V_1 V_2 V_2) + 2\eta n^3$. 
As $|V_1| |V_2| \ge (\frac1{16} - 2\r) n^2$, we have $ 2\eta n^3\le 34 \eta |V_1| |V_2| n$.
It follows that $2 e(V_1 V_1 V_2) + 2 e(V_1 V_2 V_2) \ge |V_1| |V_2| (1/2 - \r - 34 \eta) n$.  
Note that $|V_2|\le (\frac34 + 4\r)n/2$ because $|V_3|\ge |V_2| \ge |V_1|\ge (\frac14 - 4\r)n$.
Therefore 
\[
e(V_1 V_1 V_2) + e(V_1 V_2 V_2) \ge \frac12  |V_1| |V_2| \left(\frac12 - 2\r \right) n \ge \frac12  |V_1| |V_2| \, \frac54 \left(\frac{3}{8} + 2\r \right)n \ge \frac54 |V_1| \binom{|V_2|}{2}
\] 
as $\eta \ll \gamma \ll 1$. Hence
\begin{align*}
( |V_2| -1 ) e(V_1 V_1 V_2) + ( |V_1| -1 ) e(V_1 V_2 V_2)  \ge ( |V_1| -1 ) \left( e(V_1 V_1 V_2) +  e(V_1 V_2 V_2) \right)  \ge \frac52 \binom{|V_1|}{2} \binom{|V_2|}{2}.
\end{align*}
By Proposition~\ref{prop:axy}, there are at least $\frac14 \binom{|V_1|}{2} \binom{|V_2|}{2} \ge 4^{-6} n^4 $ copies of $K_4^-$ with two vertices in $V_1$ and two vertices in $V_2$.
Lemma~\ref{lem:B5} and Proposition~\ref{prop21} imply that there exist distinct $i, j \in [d]$ such that $V_{i} \cup V_{j}$ is $(c_{d-1}, \eta_{d-1})$-closed in $H$.
Also, by Proposition~\ref{prop21}, every set in $\{V_1,\dots , V_d\} \setminus \{ V_{i} , V_{j} \}$ is $(c_{d-1}, \eta_{d-1})$-closed in $H$.
Altogether this contradicts the minimality of~$d$. So $d \not \in \{3,4\}$.

\medskip
\noindent \textbf{Case 2: $d =2$.}
\medskip

By Lemma \ref{lem:C}, since $H$ is not $3\gamma$-extremal, $H$ contains at least $\gamma^2 n^3$ $V_1V_1V_2$-edges and at least $\gamma^2 n^3$ $V_1V_2V_2$-edges.
By Lemma~\ref{lem:B5}, we may assume that there are at most $ \eta n^4$ copies $K_4^-$ with two vertices in $V_1$ and two vertices in $V_2$ -- otherwise $V_{1} \cup V_{2}$ is $(c_{1}, \eta_{1})$-closed in $H$.
Thus, for all but at most $2 \sqrt{\eta} n^3$ $V_1V_1V_2$-edges~$e$, $|L(e) \cap V_2 | \le \sqrt{\eta} n $ and so $|L(e) \cap V_1 | \ge n/8 $ by Proposition~\ref{prop:le}.
Therefore, there are at least $(\r^2 n^3 - 2\sqrt{\eta} n^3) (n/8)/3 \ge \gamma^2 n^4/50 \ge \eta n^4$ copies of $K_4^-$ with three vertices in $V_1$ and one vertex in~$V_2$.
So there are at most $\eta n^4$ copies of $K_4^-$ with all vertices in~$V_1$.
(Indeed, otherwise Lemma~\ref{lem:tran} implies that $V_{1} \cup V_{2}$ is $(c_{1}, \eta_{1})$-closed in $H$, a contradiction to the minimality of $d$.)
Proposition~\ref{prop:0.3} implies that $e(H[V_1]) \le (3/10) \binom{|V_1|}3$.
Thus we have
\begin{align*}
e(V_1V_1V_2)\ge \left(\delta_2(H)  - \frac{3}{10} |V_1| \right) \binom{|V_1|}2 \ge \left( (1/2 - \gamma) n  - \frac{3}{10} |V_1| \right)  \binom{|V_1|}2.
\end{align*}
Similarly, we have $e(V_1V_2V_2) \ge ( (1/2 - \gamma) n  -   3|V_2|/10)\binom{|V_2|}2$.
So we have
\begin{align*}
&( |V_2| -1 ) e(V_1 V_1 V_2) + ( |V_1| -1 ) e(V_1 V_2 V_2) \\
\ge & \frac12(|V_1| - 1)(|V_2|-1) \left((1/2 - \gamma) n (|V_1| + |V_2|) - \frac{3}{10} (|V_1|^2 + |V_2|^2) \right) \nonumber \\
= & \frac12(|V_1| - 1)(|V_2|-1)\left( \left(1/5 - \r \right) n^2 + \frac{3}{5} |V_1| |V_2|\right).
\end{align*}
Since $|V_1| |V_2| \le n^2/4$, we have $\left(1/5 - \r \right) n^2 \ge \frac 35|V_1| |V_2| $. Thus, we get
\begin{align*}
&( |V_2| -1 ) e(V_1 V_1 V_2) + ( |V_1| -1 ) e(V_1 V_2 V_2) \ge \frac {12}5 \binom{|V_1|}{2} \binom{|V_2|}{2}.
\end{align*}
By Proposition~\ref{prop:axy}, there are at least $ \frac15 \binom{|X|}2 \binom{|Y|}2 > \eta n^4$ copies of $K_4^-$ in $H$ with two vertices in $V_1$ and two vertices in $V_2$, contradicting our assumption. So $d \not = 2$. 
This completes the proof.
\end{proof}

\subsection{Proof of Lemma~\ref{lem:B5}} \label{sec:B5}

Suppose that $H$ is a $3$-graph satisfying the hypothesis of Lemma~\ref{lem:B5}. Let $Y= V_{i_3} \cup V_{i_4}$ and $X=V(H)\setminus Y$. 
Then there are many copies of $K_4^-$ with two vertices in $X$ and two in $Y$.
Our aim is to show that there are 
many $(X,Y)$-bridges so that we can apply Lemma~\ref{lem:bridge}.

Let $H$ be a $3$-graph of order~$n$.
Let $X,Y$ be two disjoint subsets of~$V(H)$.
Given a constant $\rho>0$,  a set $x x' y $ of three vertices with $x, x'\in X$ and $y \in Y$ is called  \emph{$(\rho, X,Y)$-typical} if
\begin{itemize}
\item[(T1)] $\deg(x x', Y) \ge |Y| - \rho n$,
\item[(T2)] $|N(x y, X)\cap N(x' y, X)| \le \rho n$,
\item[(T3)] $|X| - \rho n \le \deg(x y, X) + \deg(x' y, X)$.
\end{itemize}
In the next lemma, we show that given a copy of $K_4^3$ on $xx'yy'$ with $ x,x' \in X$ and $y,y' \in Y$, then we can find many $(X,Y)$-bridges (containing this $K_4^3$) unless $xx'y$ is typical.

\begin{lemma}\label{lem:b1} Let $\gamma >0$.
Let $H$ be a $3$-graph of order~$n$ with $\delta_2(H) \ge (1/2 - \gamma)n$.
Suppose that $X, Y$ is a bipartition of $V(H)$ and that $x x' y y'$ spans a copy of $K^3 _4$ in  $H$ with $x, x'\in X$ and $y, y'\in Y$.
Then at least one of the following holds:
\begin{enumerate}[label={\rm (\alph*)}]
 \item $x x' y y'$ is contained in at least $\gamma n/4$ $5$-sets that span $(X, Y)$-bridges of length $1$; 
	\item $(1/2 - 4\r)n\le |X|, |Y|\le (1/2 + 4\r)n$ and $x x' y$ is $(9\gamma,X,Y)$-typical.
\end{enumerate}
\end{lemma}

\begin{proof}
First note that each $z\in L(x x' y)\cap X$ gives an $(X, Y)$-bridge $(z, y', \{x, x', y\})$ of length 1 because both $z x x' y$ and $x x' y y'$ span copies of $K_4^-$.
Thus we may assume that
\begin{align} \label{eq:r4}
|L(x x' y)\cap X|, |L(x x' y')\cap X|, |L(x y y')\cap Y|, |L(x' y y')\cap Y|< \gamma n/4
\end{align}
or else (a) holds. Since $xx'y\in E(H)$, this implies that $|N(x y, X) \cap N(x' y, X)| \le | L(xx'y)\cap X| \le \r n/4$.
Furthermore, by Proposition \ref{prop:le}, we have
\begin{equation} \label{eq:4in1}
\begin{split}
&\deg(x x', X) + \deg(x y, X) + \deg(x' y, X) \le |X| + \gamma n/2, \\
&\deg(x x', X) + \deg(x y', X) + \deg(x' y', X) \le |X| +  \gamma n/2,  \\
&\deg(x y, Y) + \deg(x y', Y) + \deg(y y', Y) \le |Y| +  \gamma n/2, \\
&\deg(x' y, Y) + \deg(x' y', Y) + \deg(y y', Y) \le |Y| +  \gamma n/2.
\end{split}
\end{equation}
Let $D:= 2\deg(x x', X) + 2\deg(y y', Y) + \deg(x y) + \deg(x y') + \deg(x'y) + \deg(x'y')$. 
By summing all the inequalities in \eqref{eq:4in1}, we derive that $D\le 2|X|+ 2|Y| +  2\gamma n$.
Since $\delta_2(H)\ge (1/2-\r)n$ and $|X|+ |Y|=n$, it follows that
\begin{align}\label{eq:gap}
 \deg(x x', X) + \deg(y y', Y) \le 3\r n.
\end{align}
Note that $|X| \ge \deg(y y', X) \ge (1/2 - \gamma ) n - \deg(y y', Y) \ge (1/2 - 4\gamma)n$.
Similarly we have $|Y|\ge (1/2 - 4\r)n$.
Therefore
\begin{align}
(1/2 - 4\r)n \le |X|, |Y| \le (1/2 + 4\r) n \label{eqn:b11}
\end{align}
and 
\[
	\deg(xx',Y) \overset{\eqref{eq:gap}}{\ge} (1/2 - \r)n - 3\r n \ge  |Y| - 8\r n. 
\]
If $ \deg(x x', X) + \deg(x y, X) + \deg(x' y, X) < |X| - 11 \gamma n/2$ , then we replace the first inequality in \eqref{eq:4in1} and obtain that 
\[
\deg(x y) + \deg(x y') + \deg(x'y) + \deg(x'y') \le D< 2|X|+ 2|Y| - 4\gamma n \le 4\delta_2(H),
\]
a contradiction. We thus have that $ \deg(x x', X) + \deg(x y, X) + \deg(x' y, X) \ge |X| - 11 \gamma n/2$.
Together with  \eqref{eq:gap}, it follows that $\deg(x y, X) + \deg(x' y, X) \ge |X| - 9\gamma n$. We thus deduce that $x x' y$ is $(9\gamma,X,Y)$-typical and (b) holds. 
\end{proof}

Now we prove a similar lemma in which we assume $xx'yy'$ spans a copy of $K_4^-$ -- this requires a more careful analysis of the neighbourhoods in the proof.

\begin{lemma}\label{lem:b2} Let $\gamma >0$.
Let $H$ be a $3$-graph of order~$n$ with $\delta_2(H) \ge (1/2 - \gamma)n$.
Suppose that $X, Y$ is a bipartition of $V(H)$.
Let $x,x' \in X$ and $y_1,y_2 \in Y$ such that $x y_1 y_2, x x' y_1, x x' y_2 \in E(H)$. 
Then at least one of the following holds:
\begin{enumerate}[label={\rm (\alph*)}]
	\item $x x' y_1 y_2$ is contained in at least $\gamma n/4$ $5$-sets that span $(X, Y)$-bridges of length~$1$; 
	\item there are at least $\gamma n/4$ copies $K$ of $K_4^3$ such that $|V(K) \cap X| = 2 = |V(K) \cap Y|$ and $|V(K) \cap \{x,x',y_1,y_2\}| =3$;
	\item $(1/2 - 15\r/4)n\le |X|, |Y|\le (1/2 + 15\r/4)n$ and $x x' y_1$ is $(26\gamma,X,Y)$-typical.
\end{enumerate}
\end{lemma}

\begin{proof}
We assume that neither (a) nor (b) holds.
Fix $i\in \{1,2\}$ and consider the edge $xx'y_i$.
Note that each $z\in L(x x' y_i)\cap X$ gives an $(X, Y)$-bridge $(z, y_j, \{x, x' ,y_i\})$ of length~1, where $j= 3-i$.
Thus $|L(x x' y_i)\cap X| < \gamma n/4$ or else (a) holds.
Therefore by Proposition \ref{prop:le}, we have
\begin{align}
&\deg(x x', X) + \deg(x y_i, X) + \deg(x' y_i, X) \le |X| + \gamma n/2 . \label{eq:e11}
\end{align}

On the other hand, each $z'\in N(x x', Y)\cap N(x y_i, Y) \cap N(x' y_i, Y)$ yields a copy of $K_4^3$ on $x x' y_i z'$.
Thus
\begin{align}
	 \gamma n/4 & > |N(x x', Y)\cap N(x y_i, Y) \cap N(x' y_i, Y)|   \label{eq:NY} \\
	 & \ge \deg(x x', Y) + \deg(x y_i, Y) + \deg(x' y_i, Y) - 2|Y|  \label{eq:e21}
\end{align}
or else (b) holds. 
By combining \eqref{eq:e11} and \eqref{eq:e21}, 
\begin{align}
	\deg(xx') + \deg(xy_i) + \deg( x' y_i) \le |X| + 2|Y| + 3\gamma n/4 = n + |Y| + 3\gamma n/4. \label{eq:e31}
\end{align}
Since $\delta_2(H) \ge (1/2 - \gamma)n$, it follows that $|Y|\ge (1/2 - 15\r/4)n$.

By a similar argument on the edge $xy_1y_2$, we deduce that $|X|\ge (1/2 - 15\r/4)n$.
Therefore, 
\begin{align}
(1/2 - 15\r/4)n \le |X|, |Y| \le (1/2 + 15\r/4)n. \label{eqn:b2XY}
\end{align}

We now bound $\deg(x x', Y) + \deg(x y_i, Y) + \deg(x' y_i, Y)$, $i\in \{1, 2\}$, from below.
If $\deg(x x', Y) + \deg(x y_i, Y) + \deg(x' y_i, Y)< 2 |Y| - 29 \gamma n/4$, then together with~\eqref{eq:e11} this gives that 
\begin{align*}
	\deg(xx') + \deg(xy_i) + \deg( x' y_i) & < |X| + 2|Y| - 27\gamma n/4  = n + |Y| - 27\gamma n/4 
	\overset{\eqref{eqn:b2XY}}{\le} 3(1/2 - \gamma) n, 
\end{align*}
implying that $\delta_2(H) < (1/2 - \gamma)n$, a contradiction. 
Hence, we have
\begin{align}
2|Y|- 29 \gamma n/4 \le \deg(x x', Y) + \deg(x y_i, Y) + \deg(x' y_i, Y). \label{eqn:b2eqn2}
\end{align}
For $0\le \ell \le 3$, let $n_{\ell}$ be the number of vertices of $Y$ that belong to $\ell$ of the sets $N(x x')$, $N(x y_i)$, $N(x' y_i)$.
Clearly, $n_0 + n_1 + n_2 + n_3 = |Y|$ and $n_1 + 2n_2 + 3n_3= \deg(x x', Y) + \deg(x y_i, Y) + \deg(x' y_i, Y)$.
We may assume that $n_3 \le \gamma n /4$ otherwise (b) holds.
By~\eqref{eqn:b2eqn2},
\begin{align*}
2|Y|- 29 \gamma n/4 & \le n_1 + 2n_2 + 3n_3 
= 2(n_0 + n_1 + n_2 + n_3) -2n_0 - n_1 + n_3 \\
& \le 2 |Y| -2n_0 - n_1 + \gamma n /4,
\end{align*}
which implies that $2n_0 + n_1  \le 15 \gamma n /2$.
In particular, 
\begin{align} \label{eq:Nbar}
|\overline{N}(x x', Y) \setminus (N(x y_i, Y)\cap N(x' y_i, Y))|\le n_0 + n_1 \le 15\r n/2,
\end{align}
where $\overline{N}(x x', Y) :=Y\setminus {N}(x x', Y)$. Applying \eqref{eq:Nbar} twice with 
$i=1,2$, we obtain
\[
|\overline{N}(x x', Y) \setminus (N(x y_1, Y)\cap N(x' y_1, Y))|+ |\overline{N}(x x', Y) \setminus (N(x y_2, Y)\cap N(x' y_2, Y))| \le 15\r n.
\]
If $|\overline{N}(x x', Y)|\ge  16 \gamma n $, then we get
\begin{align*}
| N(x y_1, Y) \cap N(x' y_1, Y) \cap N(x y_2, Y)\cap N(x' y_2, Y) | \ge \r n.
\end{align*}
In particular, $|N(x y_1, Y)\cap N(x y_2, Y)| \ge \r n$.
Since $N(x y_1, Y)\cap N(x y_2, Y) \subseteq L(x y_1 y_2)\cap Y$, this implies (a), a contradiction. 
Thus we assume that
\begin{align}
|\overline{N}(x x', Y)| <  16 \gamma n. \label{eqn:b2T1}
\end{align}

We further have
\begin{align*}
	\deg(xy_1,Y) + \deg(x' y_1,Y) & \le |Y| + | N(xy_1) \cap N(x'y_1) \cap Y| \\
	& \le |Y| + | N(xy_1) \cap N(x'y_1) \cap N(xx') \cap Y| + |Y\setminus N(xx')| \\
	& \overset{\mathclap{\eqref{eqn:b2T1}, \eqref{eq:NY}}}{\le} |Y| + \gamma n/4 + 16 \gamma n   \overset{\eqref{eqn:b2XY}}{\le} (1/2+20 \gamma) n  .
\end{align*}
Therefore
\begin{align}
	\deg(xy_1,X) + \deg(x' y_1,X)& \ge 2(1/2 - \gamma) n- (\deg(xy_1,Y) + \deg(x' y_1,Y)) \nonumber \\
	& \ge (1/2 - 22 \gamma) n  \overset{\eqref{eqn:b2XY}}{\ge} |X| - 26 \gamma n . \nonumber 
\end{align}
So (T3) holds (with $\rho =26\gamma$).
On the other hand, we have $|N(x y_1, X)\cap N(x' y_1, X)| \le \gamma  n/4$ because $|L(x x' y_1)\cap X| \le \gamma n / 4$ and $N(x y_1, X)\cap N(x' y_1, X) \subseteq L(x x' y_1)\cap X$.
So (T2) holds. Finally, (\ref{eqn:b2T1}) implies that
$\deg (xx',Y) \geq |Y|-16\gamma n$ so (T1) holds.
We thus deduce that $x x' y_1$ is $(26\gamma,X,Y)$-typical and (c) holds.
\end{proof}

In the next lemma, we show how to use $(\rho, X,Y)$-typical edges to find $(X,Y)$-bridges.
\begin{lemma}\label{clm:case1}
Let $0 < 1/n \ll \e' \ll \eps \ll \eta \ll \gamma \le \rho \ll 1/c$.
Suppose that $H$ is a $3$-graph of order~$n$ such that $\delta_2(H) \ge (1/2 - \gamma)n$. Let $X, Y$ be a partition $V(H)$ such that $(1/2- \rho) n \le |X| \le (1/2 + \rho) n $.
Further assume $X=V_{i_1}\cup V_{i_2}$ and $Y=V_{i_3}\cup V_{i_4}$ such that 
\begin{itemize}
\item[{\rm 1)}] either $V_{i_1} = V_{i_2} $ or $V_{i_1}\cap V_{i_2} = \emptyset$, and either
$V_{i_3} = V_{i_4} $ or $V_{i_3}\cap V_{i_4} = \emptyset$;

\item[{\rm 2)}] for each $j\in [4]$,  $V_{i_j}$ is $(c,\eta)$-closed in $H$ and has size at least $(1/4 - 4 \r) n$.
\end{itemize}
Suppose that there are at least $\eps n^4 $ copies $xx'yy'$ of $K_4^-$ in $H$ such that $x,x' \in X$, $y \in V_{i_3}$, $y' \in V_{i_4}$ 
 and $xx'y$ is $(\rho,X,Y)$-typical. Then at least one of the following holds:
\begin{enumerate}[label={\rm (\alph*)}]
	\item there are at least $\e' n^{4c+5}$ $(X, Y)$-bridges of length $c+1$ in $H$;
	\item $V_{i_1}\cap V_{i_2} = \emptyset$ and there are at least $\e' n^5$ $(V_{i_1}, V_{i_2})$-bridges of length $1$ in $H$.
\end{enumerate}
\end{lemma}

\begin{proof}
Define $\eps _0$ so that $\eps ' \ll \eps _0 \ll \eps$. 

Let $\mathcal{Y}$ be the set of pairs $(y,y')$ such that (i) $y \in V_{i_3}$ and $y' \in V_{i_4}$; and (ii) there are at least ${\eps} n^2$ pairs $x,x' \in X$ such that $xx'yy'$ spans a copy of $K_4^-$ and $xx'y$ is $(\rho,X,Y)$-typical.
Note that $|\mathcal{Y}| \ge {\eps} n^2  $ -- otherwise there are at most $|\mathcal{Y}|\binom{n}{2} + \binom{n}{2} \e n^2 < \e n^4$ copies of such $F$, a contradiction.

Fix $(y,y') \in \mathcal{Y}$.
Let $G$ be the graph on~$X$ such that $xx' \in E(G)$ if $xx'yy'$ spans a copy of $K_4^-$ and $xx'y$ is $(\rho,X,Y)$-typical.
Thus $e(G) \ge {\e} n^2$. Applying a classical result of Erd\H{o}s~\cite{Erdos}, we can find at least $\eps _0 n^6$ copies of $K_{3,3}$ in $G$. Fix a copy of~$K_{3,3}$ and label its vertices as $\{x_1, x_2, x_3,  x_1', x_2', x_3'\}$, where $x_i x_j' \in E(G)$ for all $i,j\in [3]$.

\medskip
\noindent\textbf{Claim.}
Let $W := \{ x_1,x_2,x_3, x'_1,x'_2,x'_3,y,y' \}$. Then at least one of the following holds.
\begin{itemize}
	\item[(i)] there are at least $\eps n^{4c}$ $(X,Y)$-bridges $S_0$ of length~$c+1$ with $|V(S_0) \cap W| = 5$;
	\item[(ii)] $X$ is not $(c, \eta)$-closed and  $| L(x_1x_2x_3) \cap X | \ge |X| -18 \rho n$ or $| L(x'_1x'_2x'_3) \cap X | \ge |X| -18 \rho n$.
\end{itemize}

\begin{proof}[Proof of the claim]
Let $i\in [3]$.  Since $x_i x'_1 y$ is $(\rho,X,Y)$-typical, (T2) and (T3) imply that
\begin{align}
|X| - \rho n &\le \deg(x_i y, X) + \deg(x'_1 y, X) \le |X| + \rho n. \label{eqn:T3}
\end{align}
Since $\deg(x_1 y, X) + \deg(x'_1 y, X) \le |X| + \rho n$, by swapping $x_1$ and $x'_1$ if necessary, we may assume that $\deg(x'_1 y, X) \le (|X| + \rho n)/2\le |X| - 11 \rho n $. 
For any $i \in [3]$, since $x_i x'_1 y$ is $(\rho,X,Y)$-typical, we have 
\begin{align}
	&\deg(x_i x_1', Y) \ge |Y| - \rho n \label{eqn:T1i}, \\
	&|N(x_i y, X)\cap N(x_1' y, X)| \le \rho n \nonumber .
\end{align}
Since $\deg(x_i y, X) = |N(x_i y, X)\cap N(x_1' y, X)| + | N(x_i y, X) \setminus  N(x'_1 y,X) |$, it follows that
\begin{align*}
 | N(x_i y, X) \setminus  N(x'_1 y,X) |   &\ge \deg(x_i y,X)  - \rho n  \overset{\eqref{eqn:T3}}{\ge} | X \setminus N(x'_1 y,X) | - 2 \rho n.
\end{align*}
This implies that, for distinct $i, j \in [3]$,
\begin{align}
	| N(x_i y,X) \cap N(x_j y,X)| & \ge | (N(x_i y,X) \cap N(x_j y,X)) \setminus N(x'_1 y,X) | \nonumber \\
	& \ge | N(x_i y,X) \setminus  N(x'_1 y,X) | + | N(x_j y,X) \setminus  N(x'_1 y,X) | - | X \setminus N(x'_1 y,X) | \nonumber \\
	& \ge | X \setminus N(x'_1 y,X) | - 4 \rho n
	\ge 7 \rho n. \label{eqn:xixjy}
\end{align}

Suppose that there exist $i, j\in [3]$ such that $\deg(x_ix_j,Y) \ge 4 \rho n$.
By applying \eqref{eqn:T1i} twice, we derive that
\begin{align*}
|L(x'_1 x_ix_j ) \cap Y | \ge |N(x'_1 x_i) \cap N(x'_1 x_j) \cap N(x_i x_j) \cap Y| \ge 2 \rho n.
\end{align*}
Without loss of generality, assume that $|L(x'_1 x_ix_j ) \cap V_{i_4}| \ge \rho n $.
Consider any $z \in L(x'_1 x_i x_j ) \cap V_{i_4}$.
Let $S$ be a $(y',z)$-connector of length~$c$ such that $S \cap W= \emptyset$
(there are at least $\eta n^{4c-1}/2$ choices for $S$).
Then $S \cup \{ x_i, x'_1, y',z\}$ is an $(x_j,y)$-connector of length~$c+1$ because both $H[S \cup \{ x_i, x'_1, y',z ,x_j\}] = H[x'_1x_ix_jz] \cup H[S \cup y']$ and $H[S \cup \{ x'_1,x_i, y',z ,y\}] = H[x'_1x_iyy'] \cup H[S \cup z]$ have $K_4^-$-factors.
Hence $(x_j,y,S \cup \{ x'_1,x_i, y',z\})$ is an $(X,Y)$-bridge of length~$c+1$.
Since a given $4c$-set may be partitioned into a singleton and $(4c-1)$-set in $4c$ ways, there are at least $\rho n \cdot (\eta n^{4c-1}/2) \cdot (1/4c)\ge \eps n^{4c}$ $(X,Y)$-bridges $S_0$ of length~$c+1$ such that $|V(S_0)\cap W|=5$. This proves (i) and we are done.

Now suppose that $\deg(x_ix_j,Y) < 4 \rho n$ for all distinct $i,j \in [3]$. Since $ |X| \le (1/2 + \rho) n$,  
\begin{align}
\deg(x_ix_j,X) \ge (1/2 - \gamma)n - 4 \rho n \ge |X| - 6 \rho n. \label{eqn:xixjX}
\end{align}
Thus $| L(x_1x_2x_3) \cap X | \ge |X| -18 \rho n$. If $X$ is not $(c,\eta)$-closed, then we obtain (ii); otherwise assume that $X$ is $(c,\eta)$-closed. By \eqref{eqn:xixjy} and~\eqref{eqn:xixjX},
\begin{align*}
| L(x_1x_2y) \cap X| \ge  |N(x_1y) \cap N(x_2y) \cap N(x_1x_2) \cap X| \ge \rho n.
\end{align*}
Consider any $z' \in L(x_1x_2y) \cap X$.
Let $S'$ be an $(x_1', z')$-connector of length~$c$ such that $S' \cap W= \emptyset$ (there are $\eta n^{4c-1}/2$ choices for $S$).
Then $S' \cup \{ x'_1,x_1,y,z'\}$ is an $(x_2,y')$-connector of length~$c+1$ because both $H[S' \cup \{ x'_1,x_1,x_2, y ,z'\}] = H[x_1x_2 y z'] \cup H[S' \cup x_1']$ and $H[S' \cup \{ x'_1,x_1, y,y' ,z'\}] = H[x'_1 x_1 y y'] \cup H[S' \cup z']$ have $K_4^-$-factors.
Thus $(x_2,y',S' \cup \{ x'_1,x_1,y,z' \})$ is an $(X,Y)$-bridge of length~$c+1$. In total, there are at least $\rho n \cdot (\eta n^{4c-1}/2 )\cdot(1/4c)\ge \eps n^{4c}$ $(X,Y)$-bridges $S_0$ of length~$c+1$ such that $|V(S_0)\cap W|=5$, implying (i).
\end{proof}
Now we return to the proof of the lemma.
We apply the claim for each pair $(y,y') \in \mathcal Y$ and  each copy of $K_{3,3}$ in $G$.

First assume that for \emph{some} pair $(y,y') \in \mathcal Y$, at least $\eps_0 n^6 / 2$ copies of $K_{3,3}$ in $G$ satisfy (ii) of the claim. Then $X$ is not $(c, \eta)$-closed, so $X = V_{i_1}\cup V_{i_2}$ with $V_{i_1}\cap V_{i_2} = \emptyset$. In addition, either $|L(x_1x_2x_3) \cap X| \ge |X| - 18 \rho n $ or $|L(x'_1x'_2x'_3) \cap X| \ge |X| - 18 \rho n $. This implies that at least $\eps_0 n^3 / 2$ 3-sets $S\subset X$ satisfy $|L(S)|\ge  |X| - 18 \rho n $. 
For any $v \in L(S) \cap V_{i_1} $ and $v' \in L(S) \cap V_{i_2} $, $(v,v',S)$ is a $(V_{i_1}, V_{i_2})$-bridge of length~$1$.
Recall that $|V_{i}|,|V_{i'}| \ge (1/4 - 4 \gamma) n$. 
Thus, there are at least 
$$(\eps_0 n^3 / 2) (|V_{i_1}|-18\rho n) (|V_{i_2}| - 18\rho n) \ge \eps' n^5$$
$(V_{i_1}, V_{i_2})$-bridges of length~$1$, implying (b).

The only other case to consider is when for \emph{every} $(y,y') \in \mathcal Y$, at least $\eps_0 n^6 / 2$ copies of $K_{3,3}$ in $G$ satisfy (i) of the claim. 
In this case,
for each $(y,y') \in \mathcal{Y}$, there exist at least $\eps_0 n^6 / 2$ 6-sets $W' \subseteq X$ such that there are at least
$\eps_0 n^{4c}$ $(X,Y)$-bridges $S_0$ of length~$c+1$ with $|V(S_0) \cap (W'\cup \{y, y'\})| = 5$.
By averaging, for each such $W'$, there is a $5$-subset $W_0\subset (W'\cup \{y, y'\})$ that is contained in at least $\eps_0 n^{4c}/\binom{8}{5}$ $(X,Y)$-bridges of length~$c+1$. Since there are at least  $| \mathcal{Y} |  (\eps_0 n^6) / 2 \binom{n}3$ choices of $W_0$ and a $(4c+5)$-set contains at most $\binom{4c+5}{5}$ such $W_0$,  the total number of $(X,Y)$-bridges of length $c+1$ is at least
\begin{align}
	\frac{| \mathcal{Y} |  \eps_0 n^6 / 2}{\binom{n}{3}} \cdot \frac{\eps_0 n^{4c}}{\binom{8}{5}}\cdot \frac1{\binom{4c+5}{5}} \ge \eps' n^{4c+5}, \nonumber
\end{align}
yielding (a).
\end{proof}

We are ready to prove Lemma~\ref{lem:B5}.

\begin{proof}[Proof of Lemma~\ref{lem:B5}]
Since each copy of $K_4^-$ misses one edge, after averaging (and relabeling the indices if necessary), we 
assume that there are at least $\e n^4/4$ copies of $K_4^-$ of type $V_{i_1}V_{i_2}V_{i_3}V_{i_4}$ such that all the missing edges are of type $V_{i_2} V_{i_3} V_{i_4}$. Denote by $\mathcal F$ the family of these copies of $K_4^-$.
Let $Y:= V_{i_3} \cup V_{i_4}$ and $X:=V(H)\setminus Y$.
Let $\eps', \eps''$ be such that
$$\eta' \ll \eps' \ll \eps'' \ll \eps, \eta,1/c.$$
We claim that it suffices to show that there exists $c'\le c +1$ such that there are at least $\eps' n^{4c'+1}$  $(X,Y)$-bridges of length~$c'$.
Indeed, there are at most four types of pairs $(V_i, V_j)$ such that $V_i\subseteq X$ and $V_j\subseteq Y$.
If there are at least $\eps' n^{4c'+1}$  $(X,Y)$-bridges of length~$c'$, then by averaging, there exist $i \ne j\in [d]$ such that there are at least $\eps' n^{4c'+1}/4$  $(V_i, V_j)$-bridges of length~$c'$.
By Lemma~\ref{lem:bridge} and Proposition~\ref{prop21}, we conclude that $V_i\cup V_j$ is $(3c+1, \eta')$-closed in $H$ as $2c+c' \le 3c+1$.

Since each $F= xx'yy' \in \mathcal{F}$ satisfies $x,x' \in X$, $y,y' \in Y$, $xx'y, xx'y', xyy' \in E(H)$,
we can apply Lemma~\ref{lem:b2} to $F$. Recall that $|\mathcal{F}|\ge \eps n^4/4$.
First assume that there are at least $\eps n^4/12 $ copies $F$ satisfying Lemma~\ref{lem:b2}(a), that is, each $V(F)$ is contained in $\gamma n/4$ 5-sets that span $(X, Y)$-bridges of length~$1$.
Then there are at least $\frac15 (\eps n^4/12 \cdot \r n/4) = \eps \gamma n^5/240 \ge \eps' n^5$ $(X, Y)$-bridges of length~$1$, as desired.

Second if there exist at least $\eps n^4/ 12$ copies $F\in \mathcal{F}$ that satisfy Lemma~\ref{lem:b2}(c), then we are done by Lemma~\ref{clm:case1}.

Finally, if there exist at least $\eps n^4/ 12$ copies $F\in \mathcal{F}$ that satisfy Lemma~\ref{lem:b2}(b), then there are $ ( \eps n^4/ 12 \cdot \gamma n /4) / 4n \ge \eps'' n^4$ copies $K$ of $K^3_4$ with $|V(K) \cap X| = |V(K) \cap Y|$.
Apply Lemma~\ref{lem:b1} to each such~$K$. 
First suppose there exists at least $\eps'' n^4/2 $ copies $K$ that satisfy Lemma~\ref{lem:b1}(a), that is, $V(K)$ is contained in $\gamma n/4$ 5-sets that span $(X, Y)$-bridges of length~$1$.
Then there are at least $\frac15  ( \eps'' n^4/2 \cdot \gamma n /4) \ge \eps' n^5$ $(X, Y)$-bridges of length~$1$ and we are done.
Thus we may assume that at least $\eps'' n^4/2$ copies $K = xx'yy'$ with $x,x' \in X$ and $y,y' \in Y$ satisfy Lemma~\ref{lem:b1}(b). Then we are done by Lemma~\ref{clm:case1}.
\end{proof}

\section{The extremal case}\label{secex}
In this section we prove Lemma~\ref{extremallemma}, that is,  Theorem~\ref{mainthm} in the case when $H$ is extremal.
Let $H, H'$ be two $k$-graphs on the same vertex set $V$. Let $H' \setminus H := (V, E(H')\setminus E(H))$. Suppose that $0\le \alpha\le 1$ and $|V|= n$. A vertex $v\in V$ is called \emph{$\alpha$-good in $H$} (otherwise \emph{$\alpha$-bad}) with respect to $H'$ if $\deg_{{H'}\setminus H}(v) \le \alpha n^{k-1}$. We first deal with the extremal case in the special case when every vertex in $H$ is `good'.

\begin{lemma}\label{allgood}
Let $0<1/m \ll \alpha <1/10^5$ where $m \in \mathbb N$. Suppose that $H$ is a $3$-graph on $V=A\cup B$ where $|A|=|B|=6m$. Further, suppose that every vertex in $H$ is $\alpha$-good with respect to $\mathcal B[A,B]$. Then $H$ contains a  $K_4 ^-$-factor.
\end{lemma}
\proof
Write $n:=12m$. Define an auxiliary $3$-graph $H_B$ with vertex set $B$ and where $xyz \in E(H_B)$ precisely if there are at least $|A|-3\alpha ^{1/2}n$ vertices $a \in A$ such that $xyza$ spans a copy of $K^- _4$ in $H$.

\begin{claim}
$H_B$ contains a perfect matching $M_B$.
\end{claim}
\proof
To prove the claim, consider any vertex $x \in B$. Since $x$ is $\alpha$-good, all but at most $\alpha^{1/2} n$ vertices $y \in B\setminus \{x\}$ are such that $\deg_H(xy,A) \geq |A|-\alpha ^{1/2} n$. Fix such a vertex $y \in B$. Similarly,  all but at most $2\alpha^{1/2} n$ vertices $z \in B\setminus \{x,y\}$ are such that $\deg_H(xz,A), \deg_H(yz,A) \geq |A|-\alpha ^{1/2} n$. Fix such a vertex $z \in B$. Notice that $xyza$ spans a copy of $K^- _4$ for at least $|A|-3\alpha ^{1/2} n$ vertices $a \in A$.
Thus, $xyz \in E(H_B)$.

There are at least $|B|-2\alpha ^{1/2} n$ choices for $y$ and at least $|B|-3\alpha ^{1/2} n$ choices for $z$. Hence,
$$\deg_{H_B} (x) \geq \frac{1}{2}(|B|-2\alpha ^{1/2} n)(|B|-3\alpha ^{1/2} n),$$ and so
$\delta _1(H_B) \geq \frac{1}{2}(|B|-2\alpha ^{1/2} n)(|B|-3\alpha ^{1/2} n).$ This minimum vertex degree condition forces a perfect matching in $H_B$ (for example, by a result of Daykin and H\"aggkvist~\cite{dayhag}), as required.
\endproof

By definition of $H_B$, for each edge $xyz$ in $M_B$, we can greedily pair off $xyz$ with a distinct vertex $a \in A$ so that $xyza$ forms a copy of $K^- _4$ in $H$. We therefore obtain a $K^- _4$-tiling $\mathcal M_1$ in $H$ that covers all of $B$ and $2m$ vertices in $A$. Let $A':=A\setminus V(\mathcal M_1)$. So $|A'|=4m$.

Set $H':=H[A']$. Further, define an auxiliary $4$-graph $H_{A'}$ with vertex set $A'$ and where $xyzw \in E(H_{A'})$ precisely if $xyzw$ spans a copy of $K^- _4$ in $H$.
Consider any $x\in A'$. Since $x$ is $\alpha$-good, all but at most $2{\alpha}^{1/2} n$ vertices $y\in A'\setminus \{x\}$ are such that $\deg_{H'} (xy)\geq |A'|-2-\alpha ^{1/2} n$ in $H'$. Fix such a vertex $y$. Next fix a vertex $z \in N_{H'} (xy)$ where
$\deg_{H'} (xz) \geq |A'| -2 -\alpha^{1/2} n$. There are at least $|A'|-2-3 \alpha ^{1/2}n$ choices for $z$. Finally, fix a vertex $w \in N_{H'} (xy) \cap N_{H'} (xz)$; there are at least $|A'|-4-2\alpha ^{1/2} n$ choices for $w$.
Then $xyzw$ spans a copy of $K^- _4$ in $H'$ and so $xyzw \in E(H_{A'})$.
There are at least $|A'|-3{\alpha}^{1/2} n$ choices for $y$, $|A'|-4 \alpha ^{1/2}n$ choices for $z$ and $|A'|-3\alpha ^{1/2} n$ choices for $w$. Therefore,
$$\delta _1 (H_{A'}) \geq \frac{1}{3!} (|A'|-3{\alpha}^{1/2} n) (|A'|-4 \alpha ^{1/2}n)(|A'|-3\alpha ^{1/2} n).$$
This implies that $H_{A'}$ contains a perfect matching (again, by the result of Daykin and H\"aggkvist~\cite{dayhag}) and thus, $H'$ contains a  $K^- _4$-factor $\mathcal M_2$. So $\mathcal M_1 \cup \mathcal M_2$ is a  $K^- _4$-factor in $H$, as desired. 
\endproof

We now apply Lemma~\ref{allgood} to prove Lemma~\ref{extremallemma}.

\begin{proof}[Proof of Lemma~\ref{extremallemma}]
Let $0<1/n_0 \ll \gamma \ll \gamma _1 \ll \gamma _2 \ll \gamma _3 \ll 1$. Suppose that $H$ is as in the statement of the lemma. In particular, since $H$ is $\gamma$-extremal, there exists a partition $A,B$ of $V(H)$ such that $|A|=|B|$ and $H$
$\gamma$-contains $\mathcal B[A,B]$. Furthermore, this implies that all but at most $\gamma _1 n$ vertices in $H$ are $\gamma _1$-good with respect to $\mathcal B[A,B]$. Let $A_0$ and $B_0$ denote the set of $\gamma_1$-bad vertices in $A$ and $B$ respectively. 

We say that a vertex $x \in A_0\cup B_0$ is \emph{$B$-acceptable} if there are at least $n^2/40$ pairs $(a,b)$ of vertices where $a \in A$, $b \in B$ and $abx \in E(H)$. Otherwise we say that $x \in A_0\cup B_0$ is \emph{$A$-acceptable}. Since
$\delta _1 (H) \geq (n-1)(n/2 - 1)/2 $, if $x$ is $A$-acceptable then:
\begin{itemize}
\item There are at least $3\binom{|A|}{2}/4$ pairs $(a,a')$ of vertices where $a,a' \in A$ and $aa'x \in E(H)$ and;
\item There are at least $3\binom{|B|}{2}/4$ pairs $(b,b')$ of vertices where $b,b' \in B$ and $bb'x \in E(H)$.
\end{itemize}
Next we modify the partition $A,B$ of $V(H)$ as follows. Move all $A$-acceptable vertices that lie in $B$ to $A$ and move all $B$-acceptable vertices that lie in $A$ to $B$. Since $|A_0| + |B_0| \le \gamma _1 n$, $A,B$ remains a partition of $V(H)$ where now:
\begin{itemize}
\item $n/2-\gamma _1 n \leq |A|,|B|\leq n/2+\gamma _1 n$;
\item $H$ $\gamma _2$-contains $\mathcal B[A,B]$.
\end{itemize}
Moreover, there is a partition $A_1,A_2$ of $A$ so that:
\begin{itemize}
\item[($\alpha _1$)] $|A_2|\leq \gamma _1 n$;
\item[($\alpha _2$)] Every vertex in $A_1$ is $\gamma _2$-good with respect to $\mathcal B[A,B]$;
\item[($\alpha _3$)] If $x \in A_2$ then
\begin{itemize}
\item[$\bullet$] there are at least $2\binom{|A|}{2}/3$ pairs $(a,a')$ of vertices where $a,a' \in A$ and $aa'x \in E(H)$ and;
\item[$\bullet$] there are at least $2\binom{|B|}{2}/3$ pairs $(b,b')$ of vertices where $b,b' \in B$ and $bb'x \in E(H)$.
\end{itemize}
\end{itemize}
Similarly, there is a partition $B_1,B_2$ of $B$ so that:
\begin{itemize}
\item[($\beta _1$)] $|B_2|\leq \gamma _1 n$;
\item[($\beta _2$)] Every vertex in $B_1$ is $\gamma _2$-good with respect to $\mathcal B[A,B]$;
\item[($\beta _3$)] If $x \in B_2$ then  there are at least $n^2/50$ pairs $(a,b)$ of vertices where $a \in A$, $b \in B$ and $abx \in E(H)$.
\end{itemize}

Our aim will be to find a small $K_4^-$-tiling which covers all the vertices in $A_2 \cup B_2$ so that the set of uncovered vertices, $A^*$ and $B^*$ in $A$ and $B$ respectively, are such that $|A^*|=|B^*| \equiv 0 \mod 6$.  
Then ($\alpha _2$) and ($\beta _2$) will ensure that every vertex in $H[A^* \cup B^* ]$ is $\gamma _3$-good with respect to $\mathcal B[A^*, B^*]$.
Thus, Lemma~\ref{allgood} ensures a  $K^- _4$-factor in $H[A^*\cup B^*]$ and hence  a  $K^- _4$-factor in $H$.
To guarantee that $|A^*|=|B^*| \equiv 0 \mod 6$ we will require the existence of two `parity breaking' copies of $K^- _4$ in $H$. These subgraphs will be obtained in the following claim.
We say a copy $K$ of $K^-_4$ in $H$ is of \emph{type $(i,j)$} if $K$ contains $i$ vertices from $A$ and $j$ vertices from $B$.

\begin{claim}\label{KK}
$H$ contains two copies $K,K'$ of $K^- _4$ so that one of the following conditions holds:
\begin{itemize}
\item[(i)] $K$ and $K'$ are not necessarily vertex-disjoint; $K$ is of type $(2,2)$; $K'$ is of $(3,1)$;
\item[(ii)] $K$ and $K'$ are  vertex-disjoint; $K$ and $K'$ are of type $(3,1)$;
\item[(iii)] $K$ and $K'$ are  vertex-disjoint; $K$ and $K'$ are of type $(2,2)$.
\end{itemize}
\end{claim}
\proof
To prove the claim we split our argument into two cases depending on the size of $B$.

\smallskip
{\noindent \bf Case 1: $|B|\leq n/2-1$.} 
Fix $a \in A$ and $b \in B$. Then there exists a vertex $ a' \in N_H (ab) \cap A $ since $\delta _2 (H) \geq n/2-1$ and $|B\setminus \{b\}|\leq n/2-2$.
Let $x \in A\setminus \{a,a'\}$ and $y \in B\setminus \{b\}$ be arbitrary. Again there exists a vertex $x ' \in N_H (xy) \cap A$. Suppose that $x' \not = a,a'$.
Certainly, at least two of the sets $N_H(aa'), N_H (ab), N_H(a'b)$ have an intersection of size at least $4$. Thus, $H$ contains a copy $K$ of $K^- _4$ of type $(2,2)$ or $(3,1)$ where $a,a',b \in V(K)$
and $K$ is disjoint from $xyx'$.
Moreover, at least two of the sets $N_H(xx'), N_H (xy), N_H(x'y)$ have intersection at least $5$. Thus, $H$ contains a copy $K'$ of $K^- _4$ of type $(2,2)$ or $(3,1)$ where $x,x',y \in V(K')$ and so that $K$ and $K'$ are vertex-disjoint.
Hence, $K$ and $K'$ satisfy one of (i)--(iii) as desired.

The claim is therefore satisfied in this case, unless for every $x \in A\setminus \{a,a'\}$ and $y \in B\setminus \{b\}$ we have that $ N_H (xy) \cap A \subseteq \{a,a'\}$.
Suppose there are distinct $x,x_1 \in A\setminus \{a,a'\}$ and $y, y_1 \in B\setminus \{b\}$ such that $ a\in N_H (xy) \cap A $ and $a' \in N_H (x_1y_1) \cap A $. Then as above, we obtain vertex-disjoint $K$, $K'$ that satisfy one of (i)--(iii).
This implies  we may assume that for every $x \in A\setminus \{a,a'\}$ and $y \in B\setminus \{b\}$ we have  $ N_H (xy) \cap A = \{a\}$. We now show that in this case, (i) is satisfied.

Choose any edge $a_1b_1b_2 \in E(H)$ where $a_1 \in A\setminus \{a,a'\}$ and $b_1,b_2 \in B\setminus \{b\}$; such an edge exists since $H$ $\gamma _2$-contains $\mathcal B[A,B]$. Then by assumption $aa_1b_1, aa_1b_2 \in E(H)$. Thus, $a a_1 b_1 b_2$ spans a copy
$K$ of $K^- _4$ in $H$ of type $(2,2)$. 

Since $a \in A$, ($\alpha_2$) and ($\alpha_3$) imply that there exists an edge $a xx' \in E(H)$ where $x,x' \in A\setminus \{a,a'\}$. Let $y \in B\setminus \{b\}$ be arbitrary. Then by assumption $axy, ax'y\in E(H)$. Thus, $a x x' y$ spans a copy
$K'$ of $K^- _4$ in $H$ of type $(3,1)$. Hence, the claim is satisfied in this case.

\smallskip

{\noindent \bf Case 2: $|B|\geq n/2$.} Fix $a, a' \in A$. Then since $\delta _2 (H) \geq n/2-1$, there exists a vertex $ b \in N_H (aa') \cap B $.
Let $x,y \in A\setminus \{a,a'\}$  be arbitrary. So there exists a vertex $z \in N_H (xy) \cap B$. Suppose that $b \not = z$.
At least two of the sets $N_H(aa'), N_H (ab), N_H(a'b)$ have intersection at least $4$. Thus, $H$ contains a copy $K$ of $K^- _4$ of type $(2,2)$ or $(3,1)$ where $a,a',b \in V(K)$ and $K$ is disjoint from $xyz$.
Moreover, at least two of the sets $N_H(xz), N_H (xy), N_H(yz)$ have intersection at least $5$. Thus, $H$ contains a copy $K'$ of $K^- _4$ of type $(2,2)$ or $(3,1)$ where $x,y,z \in V(K')$ and so that $K$ and $K'$ are vertex-disjoint.
Hence, $K$ and $K'$ satisfy one of (i)--(iii) as desired.

We may therefore assume that for every $x,y \in A\setminus \{a,a'\}$ we have  $ N_H (xy) \cap B = \{b\}$.
 We will show that (i) is satisfied in this case. Choose any three vertices $a_1,a_2,a_3 \in A\setminus \{a,a'\}$. Then by assumption $a_1a_2b,a_2a_3b, a_1a_3 b \in E(H)$. Thus, 
 $a_1 a_2 a_3 b$ spans a copy
$K'$ of $K^- _4$ in $H$ of type $(3,1)$. 

Since $b \in B$, ($\beta _2$) and ($\beta_3$) imply that there are vertices $x,y \in A\setminus \{a,a'\} $ and $z \in B$ such that $xzb, yzb \in E(H)$. Also, by assumption we have that   $N_H (xy) \cap B = \{b\}$.
So $x y z b$ spans a copy
$K$ of $K^- _4$ in $H$ of type $(2,2)$. 
Hence, the claim is satisfied in this case.
\endproof
The next claim will allow us to cover the vertices in $A_2 \cup B_2$ with a small $K^-_4$-tiling.
\begin{claim}\label{greed}
Let $W \subseteq V(H)$ such that $|W|\leq \gamma _2 n$. Every vertex $x \in (A_2 \cup B_2) \setminus W$ lies in a copy  $K_x$ of $K^- _4$ in $H$  of type   $(1,3)$ such that $K_x$ is disjoint from $W$. 
\end{claim}
\proof
If $x \in A_2 \setminus W$ then ($\alpha _3$) together with Mantel's theorem implies that the subgraph $L_x [B\setminus W]$ of the link graph $L_x$ contains a triangle $T$.
 Note that $T$ corresponds to a copy of   $K^- _4$ in $H$  of type $(1,3)$ that contains $x$.

Suppose that $x \in B_2 \setminus W$. Since $H$ $\gamma _2$-contains $\mathcal B[A,B]$,   ($\beta _3$) implies that there are vertices $a \in A\setminus W$ and $b, b'\in B\setminus W$ such that $abb',abx,ab'x \in E(H)$. Therefore, $x$ indeed lies in a copy   of $K^- _4$ in $H$ of type $(1,3)$ that is disjoint from $W$.
\endproof

By repeatedly applying Claim~\ref{greed}, we can obtain a $K^- _4$-tiling $\mathcal M_1$ in $H$ so that:
\begin{itemize}
\item $|\mathcal M_1| \leq 2\gamma _1 n$;
\item $\mathcal M_1$ is vertex-disjoint from $K$ and $K'$;
\item Every copy of $K^-_4$ in $\mathcal M_1$ is of type  $(1,3)$;
\item Every vertex in $(A_2 \cup B_2)\setminus (V(K) \cup V(K'))$ is covered by $\mathcal M_1$.
\end{itemize}
Let $A':=A\setminus V(\mathcal M_1)$ and $B':=B\setminus V(\mathcal M_1)$. Since $n \equiv 0 \mod 4$ we have that $|A' \cup B'|\equiv 0 \mod 4$ and so $|A'|\equiv |B'| \mod 2$. Further, $n/2-7\gamma _1 n \leq |A'|,|B'|\leq n/2+\gamma _1 n$.
Since $H$ $\gamma _2$-contains $\mathcal B[A,B]$, it is easy to see that we can greedily construct a   $K^- _4$-tiling $\mathcal M_2$ in $H$ so that:
\begin{itemize}
\item $|\mathcal M_2| \leq 8 \gamma _1 n$;
\item $\mathcal M_2$ is vertex-disjoint from $ \mathcal M_1, K,K'$;
\item Every copy of $K^-_4$ in $\mathcal M_2$ is of type $(4,0)$ or $(1,3)$;
\item $|A''|=|B''|$ where $A'':=A'\setminus V(\mathcal M_2)$ and $B'':=B'\setminus V(\mathcal M_2)$.
\end{itemize}

Note that $|A''|=| B''|\equiv 0 \mod 2$ and so (a) $|A''|= |B''|\equiv 0 \mod 6$ or; (b) $|A''|= |B''|\equiv 2 \mod 6$ or; (c) $|A''|= |B''|\equiv 4 \mod 6$.
If (a) holds we set $\mathcal M_3:= \emptyset$. Suppose that (b) holds. If Claim~\ref{KK}(i) or (iii) holds then we set  $\mathcal M_3:= \{K\}$. If Claim~\ref{KK}(ii) holds then we set $\mathcal M_3 := \{K, K', K'', K'''\}$ where $K''$ and $K'''$ are two
vertex-disjoint copies of $K^-_4$ in $H$ of type $(1,3)$ which are additionally vertex-disjoint from $\mathcal M_1, \mathcal M_2, K,K'$. (It is easy to see such $K''$ and $K'''$ exist since $H$ $\gamma _2$-contains $\mathcal B[A,B]$.) Finally, suppose that (c) holds.
If Claim~\ref{KK}(i) or (ii) holds then we set $\mathcal M_3 := \{K', K''\}$ where $K''$ is a copy of $K^-_4$ in $H$ of type $(1,3)$ which is vertex-disjoint from $\mathcal M_1, \mathcal M_2, K,K'$. If Claim~\ref{KK}(iii) holds then we set $\mathcal M_3 := \{K, K'\}$.

In every case we have chosen $\mathcal M_3$ so that $|A'''|=|B'''|\equiv 0 \mod 6$ where $A''':=A''\setminus V(\mathcal M_3)$ and $B''':=B''\setminus V(\mathcal M_3)$.
Depending on the definition of $\mathcal M_3$, $A'''$ and $B'''$ could contain vertices from $K,K'$, and thus perhaps vertices from $A_2$ and $B_2$. However, by applying Claim~\ref{greed} we can obtain a $K^-_4$-tiling $\mathcal M_4$ in $H$ so that:
\begin{itemize}
\item $|\mathcal M_4| \leq 24$;
\item $\mathcal M_4$ is vertex-disjoint from $ \mathcal M_1, \mathcal M_2, \mathcal M_3$;
\item Every copy of $K^-_4$ in $\mathcal M_4$ is of type $(4,0)$ or $(1,3)$;
\item $|A^*|=|B^*|\equiv 0 \mod 6$ where $A^*:=A'''\setminus V(\mathcal M_4)$ and $B^*:=B'''\setminus V(\mathcal M_4)$;
\item $A^* \subseteq A_1$ and $B^* \subseteq B_1$.
\end{itemize}
Note that $|A^*|,|B^*|\geq n/2-\gamma_ 2 n$. Set $H^*:=H[A^* \cup B^*]$.
By ($\alpha _2$) and ($\beta _2$) we have that every vertex in $H^*$ is $\gamma _3$-good with respect to $\mathcal B[A^*,B^*]$. Therefore, Lemma~\ref{allgood} implies that $H^*$ contains a  $K^- _4$-factor, $\mathcal M_5$.
We have that $\mathcal M_1 \cup \mathcal M_2 \cup \mathcal M_3 \cup \mathcal M_4 \cup \mathcal M_5$ is a  $K^- _4$-factor in $H$, as desired.
\end{proof}

\section{Concluding remarks}
In this paper we have determined the minimum codegree threshold that ensures a $K^- _4$-factor in a $3$-graph of sufficiently large order by the absorbing method. It is attempting to apply \cite[Theorem 2.9]{mycroft} of Keevash and Mycroft to handle the non-extremal case directly. However, since the 4-system defined in \eqref{eq:Js} is not a \emph{complex} (downward-closed system), we can not apply \cite[Theorem 2.9]{mycroft} directly.

It would also be interesting to determine the minimum \emph{vertex degree} threshold for $K^- _4$-tiling. 
\begin{conjecture}\label{conj1}
Let $n \in 4\mathbb N$ be sufficiently large. If $H$ is a $3$-graph on $n$ vertices and
$$\delta _1 (H) > \binom{3n/4-1}{2}$$
then $H$ contains a $K^- _4$-factor. 
\end{conjecture}
Note that, if true, the minimum vertex degree condition in Conjecture~\ref{conj1}
is best-possible. Indeed, consider the $3$-graph $H$ whose vertex set has a partition $X,Y$ so that $|X|=n/4+1$, $|Y|=3n/4-1$ and so that the edge set of $H$ 
consists of precisely  all those edges whose intersection with $Y$ is at least $2$.
Since $|Y|=3n/4-1< 3(n/4)$,  $H$ does not contain a $K^- _4$-factor.
Further, $\delta _1 (H) =\binom{3n/4-1}{2}$.

It is desirable to generalise the Hajnal-Szemer\'edi theorem to hypergraphs. Below we discuss the codegree tiling threshold for a family of $k$-graphs that include $K_{k+1}^k$. 
Given $k\ge 2$ and $2\le i\le k+1$, let $F_i^k$ denote the (unique) $k$-graph with $k+1$ vertices and $i$ edges (thus $F_{k+1}^k = K_{k+1}^k$). 
The results on graph tiling tell us that $\delta(n, F_i^2)= (\frac{i-1}{3} + o(1))n$ for $2\le i\le 3$.
The results in \cite{KuOs-hc, LoMa, LM1} yield that $\delta(n, F_i^3)= (\frac{i-1}{4} + o(1))n$ for $2\le i\le 4$. 
This prompts us to ask the following question:
\begin{problem}\label{pro1}
Is it true that $\delta(n, F_i^k)= (\frac{i-1}{k+1} + o(1))n$ for $2\le i\le k$, in particular $\delta(n, K_{k+1}^k)= (\frac{k}{k+1} + o(1))n$?
\end{problem}
Since $F_2^k$ is $k$-partite, the result of Mycroft \cite{Myc} confirms that $\delta(n, F_2^k)= (\frac{1}{k+1} + o(1))n$.
The methods we employed in this paper may be useful to answer Problem~\ref{pro1} in other cases.


\section*{Acknowledgements}
This research was partially carried out whilst the  authors were visiting the Institute for Mathematics and its Applications at the University of Minnesota.
The authors would like to thank the institute for the nice working environment.

\bibliographystyle{plain}

{\footnotesize \obeylines \parindent=0pt
\begin{tabular}{lll}
Jie Han      &            Allan Lo, Andrew Treglown                     &  Yi Zhao \\
Instituto de Matem\'{a}tica e Estat\'{\i}stica        &           School of Mathematics						    &  Department of Mathematics and Statistics \\
Universidade de S\~{a}o Paulo               &            University of Birmingham   					&  Georgia State University \\
S\~{a}o Paulo       &            Birmingham                          &  Atlanta \\
SP 05508-090           &            B15 2TT				&  Georgia 30303\\
Brazil                   &              UK		 			&  USA
\end{tabular}
}

{\footnotesize \parindent=0pt

\it{E-mail addresses}:
\tt{jhan@ime.usp.br}, \tt{s.a.lo@bham.ac.uk}, \tt{a.c.treglown@bham.ac.uk}, \tt{yzhao6@gsu.edu}}

\end{document}